\documentclass[a4paper,11pt]{amsart} 

\usepackage{amscd,amssymb}
\usepackage{mathrsfs}
\usepackage{bm}
\usepackage{color}
\usepackage{multirow,bigdelim}
\usepackage[all]{xy}
\usepackage{hyperref}
\usepackage{tikz}
\usepackage{enumitem}

\usepackage{geometry}
\newtheorem{thm}{Theorem}
\newtheorem{lem}[thm]{Lemma}

\newtheorem{prop}[thm]{Proposition}

\newtheorem*{ack}{Acknowledgement}
\theoremstyle{definition}
\newtheorem{defin}[thm]{Definition}

\newtheorem*{notation}{Notation}
\newtheorem*{organization}{Organization of the paper}

\theoremstyle{remark}
\newtheorem{rem}[thm]{Remark}

\numberwithin{equation}{section}

\newcommand{\A}{\mathbb{A}}

\newcommand{\F}{\mathbb{F}}

\newcommand{\bP}{\mathbb{P}}

\newcommand{\Q}{\mathbb{Q}}
\newcommand{\R}{\mathbb{R}}

\newcommand{\Z}{\mathbb{Z}}

\newcommand{\Cl}{\operatorname{Cl}}

\newcommand{\NE}{\overline{\operatorname{NE}}}

\newcommand{\Spec}{\operatorname{Spec}}
\newcommand{\Supp}{\operatorname{Supp}}

\begin{document}

\title[]{Polarized cylinders on blow-ups of weighted projective planes}


\author{}
\address{}
\curraddr{}
\email{}
\thanks{}

\author{In-Kyun Kim}
\address{June E Huh Center for Mathematical Challenges, Korea Institute for Advanced Study, 85, Hoegiro Dongdaemun-gu, Seoul 02455, Republic of Korea }
\email{soulcraw@kias.re.kr}

\author{Masatomo Sawahara}
\address{Faculty of Education, Hirosaki University, Bunkyocho 1, Hirosaki-shi, Aomori 036-8560, JAPAN}
\email{sawaharam@hirosaki-u.ac.jp}

\author{Joonyeong Won}
\address{Department of Mathematics, Ewha Womans University, 52, Ewhayeodae-gil, Seodaemun-gu, Seoul, 03760, Republic of Korea} 
\email{leonwon@ewha.ac.kr}


\date{}

\dedicatory{}

\begin{abstract}
We study polarized cylinders in certain rational surfaces arising from blow-ups of weighted projective planes. In particular, we consider the surfaces obtained by blowing up $m+4$ points in general position on the weighted projective plane $\mathbb{P}(1,1,m)$. These surfaces appear naturally as weighted hypersurfaces or quasi-smooth complete intersections. 
\end{abstract}

\maketitle
\setcounter{tocdepth}{1}

Throughout this article, all considered varieties are assumed to be algebraic and defined over
an algebraically closed field of characteristic $0$.
\section{Introduction.}
The existence of polarized cylinders in Fano varieties has attracted considerable attention in recent years due to its close connection with additive group actions on affine cones and flexibility properties of affine varieties. In particular, anticanonically polarized cylinders, namely $(-K_X)$-polar cylinders, play an important role in understanding the geometry of Fano varieties and the structure of their associated affine cones \cite{CPW16a,CPW16b,CPW17,DKKW26,KPZ11,KPZ13,KPZ14,KKW25,KKW252,MW18}. The presence of such cylinders often reflects strong geometric properties of the underlying variety and provides an effective bridge between birational geometry and affine algebraic geometry.

Let $X$ be a normal projective variety. An open subset $U \subseteq X$ is called a \emph{cylinder} if
$
U \cong \mathbb{A}^1 \times Z
$
for some variety $Z$. Let $H$ be an ample $\mathbb{Q}$-divisor on $X$. We say that a cylinder $U$ is an \emph{$H$-polar cylinder} if there exists an effective $\mathbb{Q}$-divisor $D$ on $X$ such that
$
D \sim_{\mathbb{Q}} H
$
and
$
U=X\setminus \Supp(D).
$
The study of polarized cylinders is strongly motivated by the investigation of additive group actions on affine cones.

\begin{thm}[{\cite[Theorem 3.1.9]{KPZ11}, \cite[Theorem 2.1]{KPZ14}}]
Let $X$ be a normal projective variety and let $H$ be an ample divisor on $X$. Then the affine cone
\[
\Spec \bigoplus_{n\ge0} H^0(X,\mathcal{O}_X(nH))
\]
admits a non-trivial $\mathbb{G}_a$-action if and only if $X$ contains an $H$-polar cylinder.
\end{thm}

Moreover, the geometry of polarized cylinders has applications to the study of flexibility of affine cones and additive group actions on complements of hypersurfaces; see for example \cite{PW16,Per13,KLS25,KW25,Saw24a,Saw24b,Saw24c}.

Despite their importance, $(-K_X)$-polar cylinders are rather rare. For instance, it follows from \cite{KKW} that a quasi-smooth well-formed complete intersection log del Pezzo surface
of index one admits an anticanonical polar cylinder if and only if it is a smooth complete intersection of two quadrics in $\mathbb{P}^4$. On the other hand, smooth complete intersections of two quadrics in $\mathbb{P}^4$ admit a large number of cylinders. Similar phenomena also appear in higher dimensions \cite{ST25}.

As a natural generalization of smooth complete intersections of two quadrics in \(\mathbb{P}^4\), one may consider complete intersections of two weighted hypersurfaces in the weighted projective space. These complete intersections are realized as blow-ups $S_m^{m+4}$ of the weighted projective plane \(\mathbb{P}(1,1,m)\) at \(m+4\) points in general position, where $m$ is odd. More precisely, if \(m=2k-1\), then \(S_m^{m+4}\) is isomorphic to a quasi-smooth complete intersection of two weighted hypersurfaces of degree \(2k\) in
\[
\mathbb{P}(1,1,k,k,2k-1)
\]
(see \cite{CP20,KKW}).

In contrast, when \(m\) is even, the corresponding blow-up is realized as a weighted hypersurface. Namely, if \(m=2k\), then $S_m^{m+4}$ is isomorphic to a weighted hypersurface of degree \(2k+2\) in
\[
\mathbb{P}(1,1,k,k+1)
\]
(see \cite{CP20}).

By \cite[Main Theorem]{KKW}, neither of the two classes of surfaces described above admits an anticanonical polar cylinder.
Nevertheless, as in the case of smooth cubic surfaces, it is natural to investigate the existence of cylinders after varying the polarization. From this perspective, it is worthwhile to study polarized cylinders with respect to divisors other than \(-K_X\).

To state the main theorem, we briefly recall the relevant terminology. The contraction associated with the Fujita face of \(H\) is either birational or a conic bundle. We say that \(H\) is of type \(\mathrm{B}(r_H)\) in the former case and of type \(\mathrm{C}(\ell_H)\) in the latter case. Here, \(r_H\) is the Fujita rank, \(\ell_H\) is the number of nonzero coefficients in the corresponding Fujita decomposition, and the superscript \({\mathrm{sm}}\) indicates the number of \((-1)\)-curves occurring with nonzero coefficient. Precise definitions are given in \S \ref{2-1}.

The main purpose of this paper is to establish the following theorem.
\begin{thm}\label{thm: main}
Let $m$ be an integer with $m \ge 2$, let $S$ be a normal projective surface obtained by the blow-up of the weighted projective plane $\bP(1,1,m)$ at $m+4$ points in general position (see Definition \ref{def:surface}), and let $H$ be an ample $\Q$-divisor on $S$. 
Then the following assertions hold: 
\begin{enumerate}
\item Assume that $H$ is of type ${\rm B}(r_H)$. If $r_H^{sm}>0$, then $S$ contains an $H$-polar cylinder. 
\item Assume that $H$ is of type ${\rm C}(\ell_H)$; in other words, we can write: 
\begin{align*}
K_S+ \mu_HH \sim_{\Q} aB + \sum_{i=1}^{\ell_H}a_iL_i, 
\end{align*}
where $a$ is a positive number, every $a_i$ is a non-negative number, $B$ is a $0$-curve and the union $\sum_{i=1}^{m+3}L_i$ is contractible. 
If $\ell_H^{sm}>0$ or $a>3$, then $S$ contains an $H$-polar cylinder. 
\end{enumerate}
\end{thm}
\begin{organization}
In \S \ref{2}, we review several notations for the Fujita invariant for log del Pezzo surfaces, where we follow \cite{CPW17}. 
Since the surface $S_m^{m+4}$, which is obtained by blow-up at $m+4$ general points on $\bP(1,1,m)$, is a log del Pezzo surface, the Fujita invariant theory applies in our setting. 
Hence, we can use the Fujita invariant theory to $S_m^{m+4}$. 
On the other hand, we also construct some examples of cylinders of $S_m^{m+4}$ in this section. 
These examples will be used in the next section. 

In \S \ref{3}, we prove Theorem \ref{thm: main}. 
The strategy is as follows: Let $H$ be an ample $\Q$-divisor on $S_m^{m+4}$; we note that $H$ is of type ${\rm B}$ or ${\rm C}$. 
In the type ${\rm B}$ case, we construct an $H$-polar cylinder by generalizing arguments in {\cite{Bel23}}. 
In the type ${\rm C}$ case, we observe that special situations reduce to the former case. 
In the remaining case, we construct an $H$-polar cylinder by using $\bP^1$-fibration associated with $H$ of type ${\rm C}$. 
\end{organization}
\begin{notation}
We employ the following notations: 
\begin{itemize}
\item $\A ^n$: the affine space of dimension $n$. 
\item $\bP ^n$: the projective space of dimension $n$. 
\item $\bP(a_1,\dots,a_n)$: the weighted projective space of weights $a_1,\dots,a_n$. 
\item $\F _n$: the Hirzebruch surface of degree $n$. 
\item $\A ^1_{\ast}$: the affine line with one closed point removed. 
\item $\A ^1_{\ast \ast}$: the affine line with two closed points removed. 
\item $K_X$: the canonical divisor on $X$. 
\item $D_1\sim D_2$: $D_1$ and $D_2$ are linearly equivalent. 
\item $D_1\sim _{\Q} D_2$: $D_1$ and $D_2$ are $\Q$-linearly equivalent. 
\item $(D_1 \cdot D_2)$: the intersection number of $D_1$ and $D_2$. 
\item $(D)^2$: the self-intersection number of $D$. 
\item $\varphi ^{-1}_{\ast}(D)$: the strict transform of $D$ by a morphism $\varphi$. 
\item $\psi _{\ast}(D)$: the direct image of $D$ by a morphism $\psi$. 
\item $\Supp (D)$: the support of $D$. 
\item $|D|$: the complete linear system of $D$. 
\item $I_{\mathsf{q}}(C,C')$: the intersection multiplicity of $C$ and $C'$ at $\mathsf{q}$. 
\end{itemize}
A normal projective surface $X$ with at most klt singularities is a {\it log del Pezzo surface} if its anti-canonical divisor $-K_X$ is ample. 
A morphism $\varphi: X \to Y$ from a normal projective surface $X$ to a smooth projective curve $Y$ is a {\it $\bP^1$-fibration} means general fibers of $\varphi$ are isomorphic to $\bP^1$. For an integer $a$, an {\it $a$-curve} $C$ on a normal projective surface means $(C)^2 = a$ and $C \cong \bP^1$. 
\end{notation}
\section{Preliminaries}\label{2}
\subsection{Log del Pezzo surfaces $S_m^n$}\label{2-0}
Let $m$ and $n$ be integers with $m \ge 2$. 
\begin{defin}[{\cite[Definition 2.7]{CP20}}]\label{def:surface}
The surface $S_m^n$ is the blow-up of $\bP(1,1,m)$ at $n$ distinct smooth points in general position. Here, these points are in {\it general position} if $S_m^n$ contains no irreducible curve $C$ such that $(C)^2<-1$, and, moreover, no two $(-1)$-curves on the minimal resolution $\tilde{S}_m^n$ of $S_m^n$ intersect on the exceptional curve $\tilde{Q}$ of its minimal resolution $\tilde{S}_m^n \to S_m^n$. 
\end{defin}
\begin{rem}
\leavevmode
\begin{itemize}
\item The surface $S_m^n$ has a unique singular point, which is a klt singularity of type $\frac{1}{m}(1,1)$. 
Let $\pi:\tilde{S}_m^n \to S_m^n$ be the minimal resolution, and let $\tilde{Q}$ be the exceptional curve. 
Then we know that $\tilde{Q}$ is a $(-m)$-curve. 
Moreover, every irreducible curve on $\tilde{S}_m^n$ with self-intersection number $<0$ is either $\tilde{Q}$ or a $(-1)$-curve. 
\item Note that $(-K_{S_m^n})^2 = \frac{(m+2)^2}{m}-n$. Hence, $S_m^n$ is a log del Pezzo surface if and only if either $n \le m+4$ or $n=m+5$ and $m \in \{2,3\}$ (see {\cite[Remark 2.8]{CP20}}). 
\end{itemize}
\end{rem}
\begin{defin}[{\cite[Definition 2.9]{CP20}}]
Fix $m+1$ points $p_1,\dots,p_{m+1}$ on $\bP(1,1,m)$ in general position. 
There exists a unique curve $C$ in the linear system $|\mathscr{O}_{\bP(1,1,m)}(m)|$ passing through these $m+1$ points. 
Let $S_m^{m+1} \to \bP(1,1,m)$ be the blow-up at $p_1,\dots,p_{m+1}$, and let $C'$ be the strict transform of $C$, which is a $(-1)$-curve. 
Then we denote by $\overline{S}_m^{m+1}$ the surface obtained by contracting $C'$. 
\end{defin}
\begin{rem}[{\cite[Proposition 2.8]{KL26}}]
$\overline{S}_m^{m+1}$ can be obtained by blowing-up from $\bP^2$ at $m+1$ points on a fixed line and then contracting the strict transform of this line. 
\end{rem}
The sequence of successive blow-ups, together with the contraction of the $(-1)$-curve $C'$ on $S_m^{m+1}$, is illustrated in the following diagram (see also Lemma \ref{lem: (-1)-curve (4)} below): 
\begin{align*}
\xymatrix@C=24pt@R=24pt{
\bP(1,1,m) & S_m^1 \ar[l] & \cdots \ar[l] & S_m^m \ar[l] & S_m^{m+1} \ar[l] \ar[d] & \cdots \ar[l] & S_m^{m+4} \ar[l] \\
&&&& \overline{S}_m^{m+1}&&
}
\end{align*}
We refer to the following result: 
\begin{lem}[{\cite[Lemma 2.14]{CP20}}]\label{lem: CP20}
There exists a birational morphism $\eta: \tilde{S}_m^{m+3} \to \bP^2$ such that $\eta_{\ast}(\tilde{Q})$ is an irreducible conic. 
In other words, $\tilde{S}_m^{m+3}$ is isomorphic to the blowing-up of $\bP^2$ at $m+3$ points lying on an irreducible conic. 
\end{lem}
From now on, we assume $n \le m+4$. 
By construction, there exists a birational morphism $\psi:\tilde{S}_m^n \to \F_m$. Let $\tilde{E}_1,\dots,\tilde{E}_n$ be the exceptional curves of $\psi$. Composing $\psi$ with the natural $\bP^1$-bundle $\F_m\to\bP^1$, we obtain a $\bP^1$-fibration $\varphi:\tilde{S}_m^n \to \bP^1$. See also the following commutative diagram: 
\begin{align}\label{comm diag}
\xymatrix@C=24pt@R=24pt{
\tilde{S}_m^n \ar[r]_{\psi} \ar[d]_{\pi} \ar@/^18pt/[rr]^{\varphi} & \F_m \ar[r] \ar[d] & \bP^1 \\
S_m^n \ar[r] & \bP(1,1,m) &
}
\end{align}
Notice that $\varphi$ admits exactly $n$ singular fibers $\tilde{F}_1,\dots,\tilde{F}_n$; moreover, for every $i=1,\dots,n$ there exists a $(-1)$-curve $\tilde{E}_i'$ on $\tilde{S}_m^n$ such that $\tilde{F}_i = \tilde{E}_i+\tilde{E}_i'$ and $(\tilde{E}_i' \cdot \tilde{Q}) = 1$. 
Let $\tilde{F}$ be a general fiber of $\varphi$. 
In addition, let $\tilde{L}$ be a $(-1)$-curve on $\tilde{S}_m^n$. 
\begin{lem}[{\cite[Lemma 3.3]{KLS26}}]\label{lem: (-1)-curve (1)}
With the same notation and assumptions as above, assume further that $n \not= m+4$. 
Then $\tilde{L}$ is linearly equivalent to one of the following: 
\begin{itemize}
\item $\tilde{E}_i$ $(i=1,\dots,n)$. 
\item $\tilde{F} - \tilde{E}_i$ $(i=1,\dots,n)$. 
\item $\tilde{Q} + m\tilde{F} - (\tilde{E}_{i_1} + \dots + \tilde{E}_{i_{m+1}})$ $(1 \le i_1 < i_2 < \dots < i_{m+1} \le n)$ if $n \ge m+1$. 
\item $\tilde{Q} + (m+1)\tilde{F} - (\tilde{E}_{i_1} + \dots + \tilde{E}_{i_{m+3}})$ $(1 \le i_1 < i_2 < \dots < i_{m+3} \le n)$ if $n \ge m+3$. 
\end{itemize}
In particular, $(\tilde{L} \cdot \tilde{Q}) \le 1$. 
\end{lem}
\begin{rem}
Lemma \ref{lem: (-1)-curve (1)} cannot be extended to the case where $n=m+4$. 
For instance, $\tilde{S}_2^6$ has a $(-1)$-curve, which is linearly equivalent to $2\tilde{Q} + 4\tilde{F} - (2\tilde{E}_1 + \tilde{E}_2 + \dots + \tilde{E}_6)$. 
\end{rem}
\begin{lem}\label{lem: (-1)-curve (2)}
With the same notation and assumptions as above, assume further that $n = m+4$. 
Then the following assertions hold: 
\begin{enumerate}
\item $(\tilde{L} \cdot \tilde{Q}) \le 1$. 
\item For every $i=1,\dots,m+4$, there exists a $(-1)$-curve $\tilde{E}_i''$ on $\tilde{S}_m^{m+4}$ such that $\tilde{E}_i'' \sim \tilde{Q} + (m+1)\tilde{F} - (\tilde{E}_1 + \dots +\tilde{E}_{m+4}) + \tilde{E}_i$. 
\item The $(-1)$-curves intersecting $\tilde{Q}$ on $\tilde{S}_m^{m+4}$ are only $\tilde{E}_i'$ and $\tilde{E}_i''$ for $i=1,\dots,m+4$. 
\item If $(\tilde{L} \cdot \tilde{Q}) = 0$, then $(\tilde{L} \cdot \tilde{E}_i' + \tilde{E}_i'') = 1$ for every $i=1,\dots,m+4$. 
\end{enumerate}
\end{lem}
\begin{proof}
In (1), since $\Cl(\tilde{S}_m^{m+4})=\Z[\tilde{Q}] \oplus \Z[\tilde{F}] \oplus \bigoplus_{i=1}^{m+4}\Z[\tilde{E}_i]$, we can write: 
\begin{align*}
\tilde{L} \sim a\tilde{Q} + b\tilde{F} - \sum_{i=1}^{m+4}c_i\tilde{E}_i
\end{align*}
for some integers $a,b,c_1,\dots,c_{m+4}$. 
We note $0 \le (\tilde{L} \cdot \tilde{F}) = a$ and $0 \le (\tilde{L} \cdot \tilde{E}_i) = c_i$ for $i=1,\dots,m+4$. 
Set $d := (\tilde{L} \cdot \tilde{Q}) = -ma+b$. 
Since $(\tilde{L})^2 = (\tilde{L} \cdot K_{\tilde{S}_m^{m+4}})=-1$, we have: 
\begin{align*}
-1 = -ma^2+2ab - \sum_{i=1}^{m+4}c_i^2 = (m-2)a-2b +\sum_{i=1}^{m+4}
c_i.
\end{align*}
Hence, we obtain: 
\begin{align*}
\sum_{i=1}^{m+4}c_i = (m+2)a + 2d - 1,\qquad
\sum_{i=1}^{m+4}c_i^2 = ma^2 + 2ad + 1,\qquad 
\end{align*}
by virtue of $b=ma+d$. 

Assume $a=0$. 
Then there exists a unique integer $i_0 \in \{1,\dots,m+4\}$ such that $c_{i_0} = 1$ and $c_j = 0$ for $j \not= i_0$ because $\sum_{i=1}^{m+4}c_i^2 = 1$. 
Hence, we obtain $d=1$ by using $\sum_{i=1}^{m+4}c_i = 2d - 1$. 

In what follows, we assume $a \ge 1$. 
The Cauchy–Schwarz inequality implies that: 
\begin{align}\label{ine(1)}
\begin{split}
\left(\sum_{i=1}^{m+4}c_i\right)^2 \le (m+4)\sum_{i=1}^{m+4}c_i^2 
&\iff \left\{ (m+2)a + 2d - 1\right\}^2 \le (m+4)\left\{ma^2 + 2ad + 1\right\} \\
&\iff 2ma(d-1) + 4a(a-1) + 4d^2 - 4d - (m+3) \le 0. 
\end{split}
\end{align}
Suppose $d \ge 2$. 
Then we have $2ma(d-1) \ge 2ma$ and $4d^2-4d \ge 8$. 
On the other hand, $4a(a-1) \ge 0$ holds. 
Hence, we have: 
\begin{align}\label{ine(2)}
2ma(d-1) + 4a(a-1) + 4d^2 - 4d - (m+3)
\ge 2ma + 0 + 8 - (m+3) > 0. 
\end{align}
because $a \ge 1$. The inequality (\ref{ine(2)}) contradicts (\ref{ine(1)}). 
Thus, we obtain $d \le 1$. 

In (2), recall that $\psi:\tilde{S}_m^{m+4} \to \F_m$ is a birational morphism obtained by the contraction of $\tilde{E}_1,\dots,\tilde{E}_{m+4}$. 
Set $\hat{Q} := \psi_{\ast}(\tilde{Q})$, $\hat{F} := \psi_{\ast}(\tilde{F})$, and $\hat{\sf q}_i := \psi(\tilde{E}_i)$ for $i=1,\dots,m+4$. Since 
\begin{align*}
h^0(\F_m,\mathscr{O}_{\F_m}(\hat{Q}+(m+1)\hat{F})) = m+4
\end{align*}
and $\hat{\sf q}_1,\dots,\hat{\sf q}_{m+4}$ are in general position, we have: 
\begin{align*}
\dim \left|\hat{Q}+(m+1)\hat{F} - (\hat{\sf q}_1 + \dots + \hat{\sf q}_{m+4} - \hat{\sf q}_i) \right|
 = 0.
\end{align*}
Hence, there exists a unique effective divisor $\hat{C}_i \in |\hat{Q}+(m+1)\hat{F}|$ passing through $\hat{\sf q}_1,\dots,\hat{\sf q}_{m+4}$ other than $\hat{\sf q}_i$. Note that $\hat{C}_i$ must be irreducible because $\hat{\sf q}_1,\dots,\hat{\sf q}_{m+4}$ are in general position. Let $\tilde{E}_i'' := \psi_{\ast}^{-1}(\hat{C}_i)$. Then: 
\begin{align*}
\tilde{E}_i'' \sim \tilde{Q} + (m+1)\tilde{F} - (\tilde{E}_1 + \dots +\tilde{E}_{m+4}) + \tilde{E}_i. 
\end{align*}
and $\tilde{E}_i''$ is a $(-1)$-curve by virtue of $(\tilde{E}_i'')^2 = (\tilde{E}_i'' \cdot K_{\tilde{S}_m^{m+4}}) = -1$. 

In (3) and (4), it follows from {\cite[Lemma 3.8]{KLS26}}. 
\end{proof}
Let $\tau_1: \tilde{S}_m^n \to \tilde{S}^{(1)}$ be a contraction of $\tilde{L}$. 
By Lemmas \ref{lem: (-1)-curve (1)} and \ref{lem: (-1)-curve (2)}, we know $(\tilde{L} \cdot \tilde{Q}) \le 1$, so that $\tau_{1,\ast}(\tilde{Q})$ can be contracted to a point because it is either a $(-m)$-curve or a $-(m-1)$-curve. 
Let $\pi_1: \tilde{S}^{(1)} \to S^{(1)}$ be the contraction of $\tau_{1,\ast}(\tilde{Q})$. 
Then we obtain the following commutative diagram:  
\begin{align*}
\xymatrix@C=24pt@R=24pt{
\tilde{S}_m^n \ar[r]^{\tau_1} \ar[d]_{\pi} & \tilde{S}^{(1)} \ar[d]^{\pi_1} \\
S_m^n \ar[r]^{\sigma_1} & S^{(1)}
}
\end{align*}
where $\sigma_1$ is a birational morphism obtained by contracting $\pi_{\ast}(\tilde{L})$. 
We classify $S^{(1)}$ in the following three lemmas: 
\begin{lem}\label{lem: (-1)-curve (3)}
With the same notation and assumptions as above, assume further that $n \not= m+4$ and $(\tilde{L} \cdot \tilde{Q}) = 1$. 
Then $S^{(1)} \cong S_{m-1}^{n-1}$. 
\end{lem}
\begin{proof}
Assume that $\tilde{L} \sim \tilde{F} - \tilde{E}_i$ for some $i=1,\dots,n$. 
Note that $\tau_{1,\ast}(\tilde{E}_1 + \dots + \tilde{E}_n - \tilde{E}_i)$ is a disjoint union of $(-1)$-curves. 
Hence, we obtain the contraction $\tau_2: \tilde{S}^{(1)} \to \tilde{S}^{(2)}$ of this disjoint union. 
Since $\rho(\tilde{S}^{(2)})=2$ and $(\tau_2 \circ \tau_1)_{\ast}(\tilde{Q})$ is a $-(m-1)$-curve, $\tilde{S}^{(2)}$ is isomorphic to the Hirzebruch surface of degree $m-1$ with the minimal section $(\tau_2 \circ \tau_1)_{\ast}(\tilde{Q})$. 
In particular, $\tilde{S}^{(2)} \cong \F_{m-1}$. 
Consider the following commutative diagram: 
\begin{align*}
\xymatrix@C=24pt@R=24pt{
\tilde{S}_m^n \ar[r]^{\tau_1} \ar[d]_\pi & \tilde{S}^{(1)}\ar[d]_{\pi_1} \ar[r]^{\tau_2} & \F_{m-1} \ar[d] \\
S_m^n \ar[r]^{\sigma_1} & S^{(1)} \ar[r] & \bP(1,1,m-1) 
}
\end{align*}
We know that $S^{(1)}$ is obtained by the blowing-up of $\bP(1,1,m-1)$ at $n-1$ distinct smooth points in general position. In particular, $S^{(1)} \cong S_{m-1}^{n-1}$. 

In what follows, we assume $\tilde{L} \not\sim \tilde{F} - \tilde{E}_i$ for any $i=1,\dots,n$. 
By Lemma \ref{lem: (-1)-curve (1)}, $n = m+3$ and $\tilde{L} \sim \tilde{Q} + (m+1)\tilde{F} - (\tilde{E}_1 + \dots + \tilde{E}_{m+3})$. 
For $i=1,\dots,m+2$, let $\tilde{C}_i$ be the $(-1)$-curve such that: 
\begin{align*}
\tilde{C}_i \sim \tilde{Q} + m\tilde{F} - (\tilde{E}_1 + \dots + \tilde{E}_{m+2}) + \tilde{E}_i. 
\end{align*}
Set $\tilde{D} := \sum_{i=1}^{m+2}\tilde{C}_i$. 
Then we know $(\tilde{L} \cdot \tilde{D}) = (\tilde{Q} \cdot \tilde{D}) = 0$ and the union $\tilde{D}$ is disjoint and consists only of $(-1)$-curves. 
Hence, we obtain the contraction $\tau_2: \tilde{S}^{(1)} \to \tilde{S}^{(2)}$ of this disjoint union. 
Since $\rho(\tilde{S}^{(2)})=2$ and $(\tau_2 \circ \tau_1)_{\ast}(\tilde{Q})$ is a $-(m-1)$-curve, $\tilde{S}^{(2)}$ is isomorphic to the Hirzebruch surface of degree $m-1$ with the minimal section $(\tau_2 \circ \tau_1)_{\ast}(\tilde{Q})$. 
In particular, $\tilde{S}^{(2)} \cong \F_{m-1}$. 
Thus, we have $S^{(1)} \cong S_{m-1}^{n-1}$ by a similar argument. 
\end{proof}
\begin{lem}\label{lem: (-1)-curve (4)}
With the same notation and assumptions as above, assume further that $n \not= m+4$ and $(\tilde{L} \cdot \tilde{Q}) = 0$. 
Then the following assertions hold: 
\begin{enumerate}
\item If $n = m+1$, then either $S^{(1)} \cong S_m^m$ or $S^{(1)} \cong \overline{S}_m^{m+1}$. 
\item If $n \not= m+1$, then $S^{(1)} \cong S_m^{n-1}$. 
\end{enumerate}
\end{lem}
\begin{proof}
In (1), assume that $\tilde{L} \sim \tilde{E}_i$ for some $i=1,\dots,m+1$. 
Then $S^{(1)}$ is obtained by the blowing-up of $\bP(1,1,m)$ at $m$ distinct smooth points in general position. In particular, $S^{(1)} \cong S_m^m$. 

In what follows, we assume $\tilde{L} \not\sim \tilde{E}_i$ for any $i=1,\dots,m+1$. 
Then $\tilde{L} \sim \tilde{Q} + m\tilde{F} - (\tilde{E}_1 + \dots + \tilde{E}_{m+1})$ by Lemma \ref{lem: (-1)-curve (1)}. 
By definition of $\overline{S}_m^{m+1}$, we know $S^{(1)} \cong \overline{S}_m^{m+1}$. 

In (2), if $\tilde{L} \sim \tilde{E}_i$ for some $i=1,\dots,n$, then we know $S^{(1)} \cong S_m^{n-1}$ because $S^{(1)}$ is obtained by the blowing-up of $\bP(1,1,m)$ at $n-1$ distinct smooth points in general position. 

Hence, we assume $\tilde{L} \not\sim \tilde{E}_i$ for any $i=1,\dots,n$ in what follows. 
Then $n \in\{m+2,m+3\}$ and $\tilde{L} \sim \tilde{Q} + m\tilde{F} - (\tilde{E}_{i_1} + \dots + \tilde{E}_{i_{m+1}})$ for some $i_1,\dots,i_
{m+1}$ with $1 \le i_1 < i_2 < \dots < i_{m+1} \le n$ by Lemma \ref{lem: (-1)-curve (1)}. 
We may assume $i_j=j$ for $j=1,\dots,m+1$ by swapping the subscripts. 
For $i=1,\dots,m+1$, let $\tilde{C}_i$ be the $(-1)$-curve such that: 
\begin{align*}
\tilde{C}_i \sim \tilde{Q} + m\tilde{F} - (\tilde{E}_1 + \dots + \tilde{E}_{m+1}) - \tilde{E}_n + \tilde{E}_i. 
\end{align*}
Set $\tilde{D} := \sum_{i=1}^{m+1}\tilde{C}_i$ (resp. $\tilde{D} := \sum_{i=1}^{m+1}\tilde{C}_i + \tilde{E}_{m+2}$) if $n=m+2$ (resp. $n=m+3$). 
We know $(\tilde{L} \cdot \tilde{D}) = (\tilde{Q} \cdot \tilde{D}) = 0$, and the union $\tilde{D}$ is disjoint and consists only of $(-1)$-curves. 
Hence, we obtain the contraction $\tau_2: \tilde{S}^{(1)} \to \tilde{S}^{(2)}$ of this disjoint union. 
Since $\rho(\tilde{S}^{(2)})=2$ and $(\tau_2 \circ \tau_1)_{\ast}(\tilde{Q})$ is a $(-m)$-curve, $\tilde{S}^{(2)}$ is isomorphic to the Hirzebruch surface of degree $m$ with the minimal section $(\tau_2 \circ \tau_1)_{\ast}(\tilde{Q})$. 
In particular, $\tilde{S}^{(2)} \cong \F_m$. 
Consider the following commutative diagram: 
\begin{align*}
\xymatrix@C=24pt@R=24pt{
\tilde{S}_m^n \ar[r]^{\tau_1} \ar[d]_\pi & \tilde{S}^{(1)}\ar[d]_{\pi_1} \ar[r]^{\tau_2} & \F_m \ar[d] \\
S_m^n \ar[r]^{\sigma_1} & S^{(1)} \ar[r] & \bP(1,1,m) 
}
\end{align*}
We know that $S^{(1)}$ is obtained by the blowing-up of $\bP(1,1,m)$ at $n-1$ distinct smooth points in general position. In particular, $S^{(1)} \cong S_m^{n-1}$. 
\end{proof}
\begin{lem}\label{lem: (-1)-curve (5)}
With the same notation and assumptions as above, assume further that $n = m+4$. 
Then the following assertions hold: 
\begin{enumerate}
\item If $(\tilde{L} \cdot \tilde{Q}) = 1$, then $S^{(1)} \cong S_{m-1}^{m+3}$. 
\item If $(\tilde{L} \cdot \tilde{Q}) = 0$, then $S^{(1)} \cong S_m^{m+3}$. 
\end{enumerate}
\end{lem}
\begin{proof}
In (1), note that either $\tilde{L} = \tilde{E}_i'$ or $\tilde{L} = \tilde{E}_i''$ for some $i=1,\dots,m+4$ by Lemma \ref{lem: (-1)-curve (2)} (3). 
We may assume that either $\tilde{L} = \tilde{E}_{m+4}'$ or $\tilde{L} = \tilde{E}_{m+4}''$ by swapping the subscripts.  
Set $\tilde{F}^{(1)} := \tau_{1,\ast}(\tilde{E}_{m+4}'')$ (resp. $\tilde{F}^{(1)} := \tau_{1,\ast}(\tilde{E}_{m+4}')$) when $\tilde{L} = \tilde{E}_{m+4}'$ (resp. $\tilde{L} = \tilde{E}_{m+4}''$). 
Since $\tilde{F}^{(1)}$ is a $0$-curve on $\tilde{S}^{(1)}$, we know that $|\tilde{F}^{(1)}|$ defines a $\bP^1$-fibration $\Phi_{|\tilde{F}^{(1)}|}: \tilde{S}^{(1)} \to \bP^1$, $\tau_{1,\ast}(\tilde{Q})$ becomes a bi-section of $\Phi_{|\tilde{F}^{(1)}|}$ and each component of $\tau_{1,\ast}\left(\sum_{i=1}^{m+3}(\tilde{E}_i'+\tilde{E}_i'')\right)$ is a fiber component of $\Phi_{|\tilde{F}^{(1)}|}$. 
Let $\tau_2: \tilde{S}^{(1)} \to \tilde{S}^{(2)}$ be the contraction of $\tau_{1,\ast}(\tilde{E}_1'),\dots,\tau_{1,\ast}(\tilde{E}_{m+3}')$. Then $(\tau_2 \circ \tau_1)_{\ast}(\tilde{Q})$ is a $4$-curve and $\rho(\tilde{S}^{(2)}) = 2$. 
Hence, $\tilde{S}^{(2)}$ is a smooth del Pezzo surface of degree $8$; in particular, either $\tilde{S}^{(2)} \cong \F_1$ or $\tilde{S}^{(2)} \cong \bP^1 \times \bP^1$. 
Here, we may assume that $\tilde{S}^{(2)} \cong \F_1$. 
Indeed, when $\tilde{S}^{(2)} \cong \bP^1 \times \bP^1$, we instead
contract $\tau_{1,\ast}(\tilde{E}_{m+3}'')$ rather than $\tau_{1,\ast}(\tilde{E}_{m+3}')$. 
Let $\tau_3: \tilde{S}^{(2)} \to \bP^2$ be the contraction of the unique $(-1)$-curve. 
Then we have the following diagram: 
\begin{align*}
\xymatrix@C=24pt@R=24pt{
\tilde{S}_m^{m+4} \ar[r]^{\tau_1} & \tilde{S}^{(1)}\ar[r]^{\tau_2} & {\F_1} \ar[r]^{\tau_3} & \bP^2 
}
\end{align*}
Notice that $(\tau_3 \circ \tau_2 \circ \tau_1)_{\ast}(\tilde{Q})$ is an irreducible conic on $\bP^2$. 
Hence, $\tilde{S}^{(1)} \cong \tilde{S}_{m-1}^{m+3}$ by Lemma \ref{lem: CP20} combined with the birational morphism $\tau_3 \circ \tau_2$. 

In (2), there are exactly $m+4$ $(-1)$-curves on $\tilde{S}^{(1)}$ intersecting $\tau_{1,\ast}(\tilde{Q})$ because $(\tilde{L} \cdot \tilde{E}_i' + \tilde{E}_i'') = 1$ for every $i=1,\dots,m+4$ by Lemma \ref{lem: (-1)-curve (2)} (3) and (4). 
Let $\tau_2: \tilde{S}^{(1)} \to \tilde{S}^{(2)}$ be the contraction of these $(-1)$-curves. 
Then $\tilde{S}^{(2)} \cong \bP^2$ because $\tilde{S}^{(2)}$ is a smooth projective rational surface with $(-K_{\tilde{S}^{(2)}})^2 = 9$. 
We have the following diagram: 
\begin{align*}
\xymatrix@C=24pt@R=24pt{
\tilde{S}_m^{m+4} \ar[r]^{\tau_1} & \tilde{S}^{(1)}\ar[r]^{\tau_2} & {\bP^2} 
}
\end{align*}
Since $(\tau_2 \circ \tau_1)_{\ast}(\tilde{Q})$ is an irreducible conic on $\bP^2$, we know that $\tilde{S}^{(1)} \cong \tilde{S}_m^{m+3}$ by Lemma \ref{lem: CP20} combined with the birational morphism $\tau_2$. 
\end{proof}
We finally obtain the following result: 
\begin{prop}\label{prop}
Assume that $n \le m+4$. 
Let $\pi:\tilde{S}_m^n \to S_m^n$ be the minimal resolution, let $\tilde{Q}$ be the exceptional curve of $\pi$, let $\tau': \tilde{S}_m^n \to \tilde{S}'$ be a contraction of some $(-1)$-curves such that $\tilde{Q}' := \tau'_{\ast}(\tilde{Q})$ can be contracted, and let $\pi': \tilde{S}' \to S'$ be the contraction of $\tilde{Q}'$. 
Then $S'$ is isomorphic to either $S_{m'}^{n'}$ or $\overline{S}_{m'}^{m'+1}$, where $m'$ and $n'$ are integers with $1 \le m' \le m$ and $0 \le n' < n$, $S_1^n$ is obtained by blowing-up of $\bP^2$ at $n$ points in general position, and $S_{m'}^0 := \bP(1,1,m')$. 
\end{prop}
\begin{proof}
It follows by repeatedly applying Lemmas \ref{lem: (-1)-curve (3)}, \ref{lem: (-1)-curve (4)} and \ref{lem: (-1)-curve (5)}. 
\end{proof}
\subsection{Fujita invariants}\label{2-1}
In this subsection, we summarize several notations on the Fujita invariant of log del Pezzo surfaces. 
The Fujita invariant was originally introduced by Fujita in the form of its negative, under the name of Kodaira energy ({\cite{Fuj87,Fuj92,Fuj96,Fuj97}}). 
Hassett, Tanimoto and Tschinkel defined the Fujita invariant of a log pair to study Manin’s conjecture ({\cite[Definition 2.2]{HTT15}}). 
Cheltsov, Park, and the third author established a method for applying this invariant to the study of polar cylinders on log del Pezzo surfaces ({\cite[Section 2.1]{CPW17}}).

Let $X$ be a log del Pezzo surface and let $A$ be a big $\Q$-divisor on $X$. 
By the Cone theorem, the Mori cone $\NE(X)$ of $X$ is polyhedral. 

We say that the following value: 
\begin{align*}
\mu_A := \inf \{ \lambda \in \R_{>0}\,|\, \text{the $\Q$-divisor $K_X+\lambda A$ is pseudo-effective}\}
\end{align*}
is called the {\it Fujita invariant} of the pair $(X,A)$. 
The smallest extremal face $\Delta_A$ of the Mori cone $\NE(X)$ containing $K_S+\mu_AA$ is called the {\it Fujita face} of $A$. 
Moreover, we say that $r_A := \dim \Delta_A$ is called the {\it Fujita rank} of $(X,A)$. 

Let $\phi_A:X \to Z$ be the contraction given by the Fujita face $\Delta_A$ of the divisor $A$. 
Then either $\phi_A$ is a birational morphism or a conic bundle with $Z \cong \bP^1$. 
In the former case, the $\Q$-divisor $A$ is said to be of type ${\rm B}$; and in the latter case, it is said to be of type ${\rm C}$. 
The distinction between types ${\rm B}$ and ${\rm C}$ can also be understood in terms of the Zariski decomposition of $K_X + \mu_AA$. More precisely, in Type ${\rm B}$ (resp. Type ${\rm C}$), the positive part is trivial (resp. coincides with a $0$-curve on $X$). 
\medskip

For an integer $m\geq 2$ consider a log del Pezzo surface $S := S_m^{m+4}$ obtained by blowing up the weighted projective plane $\mathbb{P}(1,1,m)$ at $m+4$ distinct smooth points $p_1,\ldots,p_{m+4}$ in general position. 
Let $H$ be an ample $\Q$-divisor on $S$. 
Recall that $S$ has a unique singular point ${\sf p}$ of type $\frac{1}{m}(1,1)$. 
\smallskip

Assume that $S$ is of type ${\rm B}$. 
Then we can write: 
\begin{align*}
K_S + \mu_H H \sim_{\Q} \sum_{i=1}^{r_H}a_iL_i,
\end{align*}
where every $a_i$ is a positive number and the union $\sum_{i=1}^{r_H}L_i$ is contractible. 
More precisely, each $L_i$ is one of the following: 
\begin{itemize}
\item $L_i$ is a $(-1)$-curve; i.e., $(L_i)^2 = (L_i \cdot K_S) = -1$. In this case, we have $0<a_i<1$. 
\item $L_i$ passes through the singular point on $S$. Moreover, $(L_i)^2 = -\frac{m-1}{m}$ and $(L_i \cdot K_S) = -\frac{2}{m}$. In this case, we have $0<a_i<\frac{2}{m-1}$. 
\end{itemize}

We say that $H$ is of {\it type ${\rm B}(r_H)$} in what follows. 
We define
\begin{align*}
r_H^{sm} := \# \{ i \in \{ 1,\dots ,r_H\}\,|\, \text{$L_i$ is a $(-1)$-curve}\}.
\end{align*} 
Since $\sum_{i=1}^{r_H}L_i$ is contractible, we know $0 \le r_H \le m+4$. 
Moreover, $0 \le r_H-r_H^{sm} \le m-1$ because there are at most $m-1$ contractible curves passing through ${\sf p}$. 

\smallskip

Assume that $S$ is of type ${\rm C}$. 
Then the Fujita rank $r_H$ of $H$ is equal to $\rho(S)-2 = m+3$ and we can write: 
\begin{align*}
K_S + \mu_H H \sim_{\Q} aB+\sum_{i=1}^{m+3}a_iL_i,
\end{align*}
where $a$ is a positive number, every $a_i$ is a non-negative number, $B$ is a $0$-curve; i.e., $(B)^2=0$ and $(B \cdot K_S) = -2$, and the union  $\sum_{i=1}^{m+3}L_i$ is contractible. 
More precisely, each $L_i$ is one of the following: 
\begin{itemize}
\item $L_i$ is a $(-1)$-curve; i.e., $(L_i)^2 = (L_i \cdot K_S) = -1$. In this case, we have $0 \le a_i<1$. 
\item $L_i$ passes through the singular point on $S$. Moreover, $(L_i)^2 = -\frac{m-1}{m}$ and $(L_i \cdot K_S) = -\frac{2}{m}$. In this case, we have $0 \le a_i<\frac{2}{m-1}$. 
\end{itemize}

Now, $\ell_H$ denotes the number $a_i$'s with $a_i \not=0$. 
Moreover, we say that $H$ is of {\it type ${\rm C}(\ell_H)$} in what follows. 
Then we shall determine singular fibers of a $\bP^1$-fibration $\phi_H:S \to \bP^1$. 
Notice that ${\sf p}$ must be included in a singular fiber of $\phi_H$, say $B_1$. 
Letting $\pi:\tilde{S} \to S$ be the minimal resolution and let $\tilde{Q}$ be the exceptional curve, which is a $(-m)$-curve, every irreducible curve on $\tilde{S}$ with negative self-intersection number is either $\tilde{Q}$ or a $(-1)$-curve. 
Moreover, since $\phi_H$ is a $\bP^1$-fibration, so is $\phi_H \circ \pi:\tilde{S} \to \bP^1$. 
Hence, $\tilde{B}_1 := \pi^{-1}_{\ast}(B_1)$ consists of $\tilde{Q}$ and exactly $m$ $(-1)$-curves, and every singular fiber other than $\tilde{B}_1$ of $\phi_H \circ \pi$ is degenerate conic because it must consist only of several $(-1)$-curves. 
Note that there are exactly four singular fibers, which are degenerate conics, of $\phi_H \circ \pi$ because the total number of irreducible components of singular fibers of $\phi_H \circ \pi$ is equal to 
$\rho(\tilde{S}) - 2 + \#\{ \text{singular fibers of}\ \phi_H \circ \pi\}$. 
Hence, $\phi_H$ has exactly five singular fibers, one passes through the singular point ${\sf p}$ and the others are degenerate conics. 
We define
\begin{align*}
\ell_H^{sm} := \# \{ i \in \{1,\dots ,m+3\}\,|\, \text{$a_i \not= 0$ and $L_i$ is a $(-1)$-curve}\}. 
\end{align*}
Since $\sum_{i=1}^{\ell_H}L_i$ is contractible, we know $0 \le \ell_H^{sm} \le 4$. 
\subsection{Examples of cylinders}\label{2-2}
Let $S := S_m^{m+4}$ be a log del Pezzo surface obtained by the blow-up of the weighted projective plane $\bP(1,1,m)$ at $m+4$ points in general position, where $m \ge 2$. 
In this subsection, we will present several examples of cylinders in $S$. 

Note that $S$ has a unique singular point ${\sf p}$ of type $\frac{1}{m}(1,1)$. 
Let $\pi:\tilde{S} \to S$ be the minimal resolution at ${\sf p} \in S$, and let $\tilde{Q}$ be the reduced exceptional curve of $\pi$. 
Note that $\tilde{Q}$ is a single $(-m)$-curve on $\tilde{S}$. 
By construction, there exists a birational morphism $\psi:\tilde{S} \to \F_m$. Let $\tilde{E}_1,\dots,\tilde{E}_{m+4}$ be the exceptional curves of $\psi$. Composing $\psi$ with the natural $\bP^1$-bundle $\F_m\to\bP^1$, we obtain a $\bP^1$-fibration $\varphi:\tilde{S} \to \bP^1$ such that $\tilde{Q}$ is a section. See also the commutative diagram (\ref{comm diag}). 
Notice that $\varphi$ admits exactly $m+4$ singular fibers $\tilde{F}_1,\dots,\tilde{F}_{m+4}$; moreover, for every $i=1,\dots,m+4$ there exists a $(-1)$-curve $\tilde{E}_i'$ on $\tilde{S}$ such that $\tilde{F}_i = \tilde{E}_i+\tilde{E}_i'$ and $(\tilde{E}_i' \cdot \tilde{Q}) = 1$. 
Let $\tilde{F}$ be a general fiber of $\varphi$. 
\begin{lem}\label{curve}
With the same notations as above, there exist a closed point $\tilde{\sf q}$ on $\tilde{Q}$ and a $0$-curve $\tilde{\Gamma}_{\tilde{\sf q}}$ such that: 
\begin{align*}
\tilde{\Gamma}_{\tilde{\sf q}} \sim \tilde{Q} + (m+2)\tilde{F} - \sum_{i=1}^{m+4}\tilde{E}_i
\end{align*}
and $I_{\tilde{\sf q}}(\tilde{\Gamma}_{\tilde{\sf q}},\tilde{Q}) = 2$. 
\end{lem}
\begin{proof}
We consider the divisor $\tilde{\Delta} := \tilde{Q} + (m+2)\tilde{F} - \sum_{i=1}^{m+4}\tilde{E}_i$ on $\tilde{S}$. 
By the Riemann-Roch theorem combined with rationality of $\tilde{S}$, we have: 
\begin{align*}
\dim |\tilde{\Delta}| \ge h^0(\tilde{S},\mathscr{O}_{\tilde{S}}(\tilde{\Delta})) -1 = \frac{1}{2}(\tilde{\Delta} \cdot \tilde{\Delta}-K_{\tilde{S}}) = 1. 
\end{align*}
Recall that $\psi:\tilde{S} \to \F_m$ is a birational morphism obtained by the contraction of $\tilde{E}_1,\dots,\tilde{E}_{m+4}$. 
Set $\hat{Q} := \psi_{\ast}(\tilde{Q})$, $\hat{F} := \psi_{\ast}(\tilde{F})$, and $\hat{\sf q}_i := \psi(\tilde{E}_i)$ for $i=1,\dots,m+4$. Since 
\begin{align*}
h^0(\F_m,\mathscr{O}_{\F_m}((m+2)\hat{F})) = m+3
\end{align*}
and $\hat{\sf q}_1,\dots,\hat{\sf q}_{m+4}$ are in general position, we know: 
\begin{align*}
\left|(m+2)\hat{F} - (\hat{\sf q}_1 + \dots + \hat{\sf q}_{m+4})\right|
 = \emptyset. 
\end{align*}
This implies that $\tilde{Q}$ is not a fixed component of $|\tilde{\Delta}|$. 
Since $(\tilde{\Delta} \cdot \tilde{Q}) = 2$, the restriction of $|\tilde{\Delta}|$ to $\tilde{Q}$ defines a linear system of degree $2$ on $\tilde{Q} \cong \bP^1$. 
Hence, there exist a closed point $\tilde{\sf q} \in \tilde{Q}$ and a member $\tilde{\Gamma}_{\tilde{\sf q}} \in |\tilde{\Delta}|$ such 
that $I_{\tilde{\sf q}} (\tilde{\Gamma}_{\tilde{\sf q}},\tilde{Q})=2$. 

Suppose that $\tilde{\Gamma}_{\tilde{\sf q}}$ is reducible. 
Then it is the sum of two $(-1)$-curves. 
Note that every $(-1)$-curve on $\tilde{S}$ intersects $\tilde{Q}$ at most once by Lemma \ref{lem: (-1)-curve (2)} (1). 
Hence, both $(-1)$-curves must intersect $\tilde{Q}$ at $\tilde{\sf q}$, since $(\tilde{\Gamma}_{\tilde{\sf q}}\cdot\tilde{Q})=2$. 
This is a contradiction to the general position assumption (see Definition \ref{def:surface}). 
Therefore, $\tilde{\Gamma}_{\tilde{\sf q}}$ is irreducible. 
In particular, we know that $\tilde{\Gamma}_{\tilde{\sf q}}$ is a $0$-curve. 
\end{proof}
Let $\tilde{\sf q}$ and $\tilde{\Gamma}_{\tilde{\sf q}}$ be the same as in Lemma \ref{curve}. 
Take a fiber $\tilde{F}_{\tilde{\sf q}}$ of $\varphi$ passing through $\tilde{\sf q}$. 
The configuration looks like that in Figure \ref{fig(1)}. 
Set $\Gamma_{\tilde{\sf q}} := \pi_{\ast}(\tilde{\Gamma}_{\tilde{\sf q}})$, $F_{\tilde{\sf q}} := \pi_{\ast}(\tilde{F}_{\tilde{\sf q}})$, $E_i := \pi_{\ast}(\tilde{E}_i)$ and $E_i' := \pi_{\ast}(\tilde{E}_i')$ for $i=1,\dots,m+4$. 
\begin{figure}[t]
\begin{center}
\begin{tikzpicture}[scale=0.7]
\draw[very thick] (0,0) -- (15,0);
\draw[thick] (2.5,-1) -- (2.5,6);
\node at (-.4,5.5) {$\tilde{\Gamma}_{\tilde{\sf q}}$};
\node at (-.4,0) {$\tilde{Q}$};
\node at (2.5,6.25) {$\tilde{F}_{\tilde{\sf q}}$};
\node at (2.75,-.35) {${\tilde{\sf q}}$};
\fill (2.5,0) circle (2pt);

\draw [thick] (0,5.5) 
.. controls (1,5.5) and (1,.1) .. (2.5,0)
.. controls (4,.1) and (4,5.5) .. (5,5.5)
.. controls (5,5.5) and (12,5.5) .. (15,5.5);

\draw (7.5,-1) -- (6.5,3);
\draw (7.5,6) -- (6.5,2);
\draw (9,-1) -- (8,3);
\draw (9,6) -- (8,2);
\draw (14,-1) -- (13,3);
\draw (14,6) -- (13,2);
\node at (10.75,2.5) {\large $\cdots \cdots \cdots \cdots$};
\node at (7.5,-1.45) {$\tilde{E}_1'$};
\node at (9,-1.45) {$\tilde{E}_2'$};
\node at (14,-1.45) {$\tilde{E}_{m+4}'$};
\node at (7.5,6.35) {$\tilde{E}_1$};
\node at (9,6.35) {$\tilde{E}_2$};
\node at (14,6.35) {$\tilde{E}_{m+4}$};
\end{tikzpicture}
\caption{Configuration of the $\bP^1$-fibration $\varphi :\tilde{S} \to \bP^1$}\label{fig(1)}
\end{center}
\end{figure}
\begin{lem}\label{lem(1)}
With the same notations as above, we obtain: 
\begin{align*}
S \setminus \Supp \left( \Gamma_{\tilde{\sf q}} + F_{\tilde{\sf q}} + \sum_{i=1}^tE_i + \sum_{j=t+1}^{m+4}E_j' \right) \cong \A^1 \times \A^1_{\ast}
\end{align*}
for every $t=2,\dots,m+4$, where we consider $\sum_{j=t+1}^{m+4}E_j' = 0$ if $t=m+4$. 
\end{lem}
\begin{proof}
Let $f:\bar{S} \to \tilde{S}$ be blow-ups at $\tilde{\sf q}$ and first infinitely near point of $\tilde{\sf q}$, and let $\bar{L}_1+\bar{L}_2$ be the reduced exceptional divisor of $f$, where $\bar{L}_1$ and $\bar{L}_2$ are a $(-2)$-curve and a $(-1)$-curve on $\bar{S}$, respectively. 
Set $\bar{\Gamma} := f^{-1}_{\ast}(\tilde{\Gamma}_{\tilde{\sf q}})$, $\bar{F} := f^{-1}_{\ast}(\tilde{F}_{\tilde{\sf q}})$, $\bar{Q} := f^{-1}_{\ast}(\tilde{Q})$, $\bar{E}_i := f^{-1}_{\ast}(\tilde{E}_i)$ for $i=1,\dots ,t$, and $\bar{E}_j' := f^{-1}_{\ast}(\tilde{E}_j')$ for $j=t+1,\dots,m+4$. 
Since $(\bar{F} \cdot \bar{L}_1) =1$ and $\bar{F}$ is a $(-1)$-curve on $\bar{S}$, we know that $\bar{F}+\bar{L}_1$ can be smoothly contracted. 
Moreover, the union $\sum_{i=1}^t\bar{E}_i + \sum_{j=t+1}^{m+4}\bar{E}_j'$ is disjoint and contractible. 
By contracting this disjoint union and $\bar{F}+\bar{L}_1$, we obtain the birational morphism $g:\bar{S} \to \F_{t-2}$. 
Then $g_{\ast}(\bar{Q})$, $g_{\ast}(\bar{L}_2)$, and $g_{\ast}(\bar{\Gamma})$ are a $-(t-2)$-curve, a $0$-curve, and a $(t-2)$-curve, respectively. 
See also Figure \ref{fig(2)}. 
Furthermore, since the pair $(\F_{t-2}, g_{\ast}(\bar{Q})+g_{\ast}(\bar{L}_2)+g_{\ast}(\bar{\Gamma}))$ is a minimal normal compactification of $\A^1 \times \A^1_{\ast}$ (see also {\cite{Koj02}}), we obtain $\F_{t-2} \setminus \Supp \left( g_{\ast}(\bar{Q})+g_{\ast}(\bar{L}_2)+g_{\ast}(\bar{\Gamma}) \right) \cong \A^1 \times \A^1_{\ast}$. 
Hence, we have: 

\begin{eqnarray*}
\lefteqn{S \setminus \Supp \left( \Gamma_{\tilde{\sf q}} + F_{\tilde{\sf q}} + \sum_{i=1}^tE_i + \sum_{j=t+1}^{m+4}E_j' \right)}\\
& &\cong \tilde{S} \setminus \Supp \left( \tilde{Q} + \tilde{\Gamma}_{\tilde{\sf q}} + \tilde{F}_{\tilde{\sf q}} + \sum_{i=1}^t\tilde{E}_i + \sum_{j=t+1}^{m+4}\tilde{E}_j' \right) \\
& &\cong \bar{S} \setminus \Supp \left( \bar{Q} + \bar{L}_1 + \bar{L}_2 + \bar{\Gamma} + \bar{F} + \sum_{i=1}^t\bar{E}_i + \sum_{j=t+1}^{m+4}\bar{E}_j' \right) \\
& &\cong \F_{t-2} \setminus \Supp \left( g_{\ast}(\bar{Q})+g_{\ast}(\bar{L}_2)+g_{\ast}(\bar{\Gamma}) \right) \\
& &\cong \A^1 \times \A^1_{\ast}. 
\end{eqnarray*}

This completes the proof. 
\end{proof}
\begin{figure}[t]
\begin{center}
\begin{tikzpicture}[scale=0.35]
\draw (0,0) -- (13,0);
\draw (2.5,-1) -- (2.5,6);
\node at (.35,6.35) {\footnotesize $\tilde{\Gamma}_{\tilde{\sf q}}$};
\node at (.35,-1) {\footnotesize $\tilde{Q}$};
\node at (2.5,6.35) {\footnotesize $\tilde{F}_{\tilde{\sf q}}$};
\draw (0,5.5) 
.. controls (1,5.5) and (1,.1) .. (2.5,0)
.. controls (4,.1) and (4,5.5) .. (5,5.5)
.. controls (5,5.5) and (12,5.5) .. (13,5.5);
\draw (6,6) -- (5.5,3.5);
\node at (6.75,4) {\footnotesize $\cdots$};
\draw (8.5,6) -- (8,3.5);
\draw (10,-1) -- (9.5,1.5);
\node at (10.75,1) {\footnotesize $\cdots$};
\draw (12.5,-1) -- (12,1.5);
\node at (6,6.35) {\footnotesize $\tilde{E}_1$};
\node at (8.5,6.35) {\footnotesize $\tilde{E}_t$};
\node at (10,-1.85) {\footnotesize $\tilde{E}_{t+1}'$};
\node at (12.5,-1.85) {\footnotesize $\tilde{E}_{m+4}'$};

\node at (14.5,3) {$\overset{f}{\longleftarrow}$};

\draw (16,0) -- (29,0);
\draw (16,5) -- (29,5);
\draw (16,2.5) -- (20,2.5);
\draw (17,-1) -- (17,6);
\draw (18.5,1.5) -- (18.5,3.5);
\draw (22,3.5) -- (22,6);
\node at (23.25,4) {\footnotesize $\cdots$};
\draw (24.5,3.5) -- (24.5,6);
\draw (25.5,1.5) -- (25.5,-1);
\node at (26.75,1) {\footnotesize $\cdots$};
\draw (28,1.5) -- (28,-1);
\node at (20.75,-1) {\footnotesize $\bar{Q}$};
\node at (19.5,5.85) {\footnotesize $\bar{\Gamma}$};
\node at (20.45,2.5) {\footnotesize $\bar{L}_1$};
\node at (17,6.35) {\footnotesize $\bar{L}_2$};
\node at (18.5,4.15) {\footnotesize $\bar{F}$};
\node at (22,6.35) {\footnotesize $\bar{E}_1$};
\node at (24.5,6.35) {\footnotesize $\bar{E}_t$};
\node at (25.5,-1.85) {\footnotesize $\bar{E}_{t+1}'$};
\node at (28,-1.85) {\footnotesize $\bar{E}_{m+4}'$};

\node at (30.5,3) {$\overset{g}{\longrightarrow}$};

\draw (33,0) -- (41,0);
\draw (33,5) -- (41,5);
\draw (34.5,-1) -- (34.5,6);
\node at (38.5,5.85) {\footnotesize $g_{\ast}(\bar{\Gamma})$};
\node at (38.5,-1) {\footnotesize $g_{\ast}(\bar{Q})$};
\node at (34.5,6.35) {\footnotesize $g_{\ast}(\bar{L}_2)$};

\end{tikzpicture}
\caption{Birational map $g \circ f^{-1}:\tilde{S} \dashrightarrow \F_{t-2}$}\label{fig(2)}
\end{center}
\end{figure}
\begin{lem}\label{curve(2)}
With the same notations as above, there exists a $0$-curve $\tilde{C}_{\tilde{\sf q}}$ on $\tilde{S}$ passing through $\tilde{\sf q}$ such that: 
\begin{align*}
\tilde{C}_{\tilde{\sf q}} \sim \tilde{Q} + (m+1)\tilde{F} - \sum_{i=3}^{m+4}\tilde{E}_i. 
\end{align*}
\end{lem}
\begin{proof}
This lemma can be shown by a similar argument to Lemma \ref{curve}. 
\end{proof}
Let $\tilde{C}_{\tilde{\sf q}}$ be the same as in Lemma \ref{curve(2)}. 
Put $C_{\tilde{\sf q}} := \pi_{\ast}(\tilde{C}_{\tilde{\sf q}})$. 
\begin{lem}\label{lem(2)}
With the same notations as above, we obtain: 
\begin{align*}
S \setminus \Supp \left( \Gamma_{\tilde{\sf q}} + F_{\tilde{\sf q}} + C_{\tilde{\sf q}} + E_1 + E_2 + \sum_{i=3}^{m+4}E_i' \right) \cong \A^1 \times \A^1_{\ast \ast}. 
\end{align*}
\end{lem}
\begin{proof}
Let $f:\bar{S} \to \tilde{S}$ be blow-ups at $\tilde{\sf q}$ and first infinitely near point of $\tilde{\sf q}$, and let $\bar{L}_1+\bar{L}_2$ be the reduced exceptional divisor of $f$, where $\bar{L}_1$ and $\bar{L}_2$ are a $(-2)$-curve and a $(-1)$-curve on $\bar{S}$, respectively. 
Set $\bar{\Gamma} := f^{-1}_{\ast}(\tilde{\Gamma}_{\tilde{\sf q}})$, $\bar{F} := f^{-1}_{\ast}(\tilde{F}_{\tilde{\sf q}})$, $\bar{C} := f^{-1}_{\ast}(\tilde{C}_{\tilde{\sf q}})$, $\bar{Q} := f^{-1}_{\ast}(\tilde{Q})$, $\bar{E}_1 := f^{-1}_{\ast}(\tilde{E}_1)$, $\bar{E}_2 := f^{-1}_{\ast}(\tilde{E}_2)$, and $\bar{E}_i' := f^{-1}_{\ast}(\tilde{E}_i')$ for $i=3,\dots,m+4$. 
Then the union $\bar{F} + \bar{C} + \bar{E}_1 + \bar{E}_2 + \sum_{i=3}^{m+4}\bar{E}_i'$ is disjoint and contractible. 
By contracting this disjoint union, we obtain the birational morphism $g:\bar{S} \to \F_1$. 
Then $g_{\ast}(\bar{Q})$, $g_{\ast}(\bar{\Gamma})$ and $g_{\ast}(\bar{L}_1)$ are $0$-curves, and $g_{\ast}(\bar{L}_2)$ is a $(-1)$-curve. 
See also Figure \ref{fig(3)}. 
Furthermore, since the pair $(\F_1, g_{\ast}(\bar{Q}) + g_{\ast}(\bar{\Gamma}) + g_{\ast}(\bar{L}_1) + g_{\ast}(\bar{L}_2))$ is a minimal normal compactification of $\A^1 \times \A^1_{\ast \ast}$ (see also {\cite{Koj02}}), we obtain $\F_1 \setminus \Supp \left( g_{\ast}(\bar{Q}) + g_{\ast}(\bar{\Gamma}) + g_{\ast}(\bar{L}_1) + g_{\ast}(\bar{L}_2) \right) \cong \A^1 \times \A^1_{\ast \ast}$. 
Hence, we have: 
\begin{align*}
&S \setminus \Supp \left( \Gamma_{\tilde{\sf q}} + F_{\tilde{\sf q}} + C_{\tilde{\sf q}} + E_1 + E_2 + \sum_{i=3}^{m+4}E_i' \right) \\
&\qquad \cong \tilde{S} \setminus \Supp \left( \tilde{Q} + \tilde{\Gamma}_{\tilde{\sf q}} + \tilde{F}_{\tilde{\sf q}} + \tilde{C}_{\tilde{\sf q}} + \tilde{E}_1 + \tilde{E}_2 + \sum_{i=3}^{m+4}\tilde{E}_i' \right) \\
&\qquad \cong \bar{S} \setminus \Supp \left( \bar{Q} + \bar{\Gamma} + \bar{F} + \bar{C} + \bar{L}_1 + \bar{L}_2 + \bar{E}_1 + \bar{E}_2 + \sum_{i=3}^{m+4}\bar{E}_i' \right) \\
&\qquad \cong \F_1 \setminus \Supp \left( g_{\ast}(\bar{Q}) + g_{\ast}(\bar{\Gamma}) + g_{\ast}(\bar{L}_1) + g_{\ast}(\bar{L}_2) \right) \\
&\qquad \cong \A^1 \times \A^1_{\ast \ast}. 
\end{align*}
This completes the proof. 
\end{proof}
\begin{figure}[t]
\begin{center}
\begin{tikzpicture}[scale=0.35]
\draw (0,0) -- (13,0);
\draw (2.5,-1) -- (2.5,6);
\node at (.35,6.35) {\footnotesize $\tilde{\Gamma}_{\tilde{\sf q}}$};
\node at (.35,.85) {\footnotesize $\tilde{Q}$};
\node at (2.5,6.35) {\footnotesize $\tilde{F}_{\tilde{\sf q}}$};
\draw (0,5.5) 
.. controls (1,5.5) and (1,.1) .. (2.5,0)
.. controls (4,.1) and (4,5.5) .. (5,5.5)
.. controls (5,5.5) and (12,5.5) .. (13,5.5);
\draw (6,6) -- (5.5,3.5);
\draw (7.5,6) -- (7,3.5);
\draw (9,-1) -- (8.5,1.5);
\node at (10.25,1) {\footnotesize $\cdots$};
\draw (12.5,-1) -- (12,1.5);
\draw (0,-1) -- (13,4.2);
\node at (10.5,4.2) {\footnotesize $\tilde{C}_{\tilde{\sf q}}$};
\node at (6,6.35) {\footnotesize $\tilde{E}_1$};
\node at (7.5,6.35) {\footnotesize $\tilde{E}_2$};
\node at (9,-1.85) {\footnotesize $\tilde{E}_3'$};
\node at (12.5,-1.85) {\footnotesize $\tilde{E}_{m+4}'$};

\node at (14.5,3) {$\overset{f}{\longleftarrow}$};

\draw (16,0) -- (29,0);
\draw (16,5) -- (29,5);
\draw (16,2.5) -- (21.5,2.5);
\draw (17,-1) -- (17,6);
\draw (18.5,1.5) -- (18.5,3.5);
\draw (20,1.5) -- (20,3.5);
\draw (23,3.5) -- (23,6);
\draw (24.5,3.5) -- (24.5,6);
\draw (25.5,1.5) -- (25.5,-1);
\node at (26.75,1) {\footnotesize $\cdots$};
\draw (28,1.5) -- (28,-1);
\node at (23,.85) {\footnotesize $\bar{Q}$};
\node at (21,5.85) {\footnotesize $\bar{\Gamma}$};
\node at (21.95,2.5) {\footnotesize $\bar{L}_1$};
\node at (17,6.35) {\footnotesize $\bar{L}_2$};
\node at (18.5,4.15) {\footnotesize $\bar{F}$};
\node at (20,4.15) {\footnotesize $\bar{C}$};
\node at (23,6.35) {\footnotesize $\bar{E}_1$};
\node at (24.5,6.35) {\footnotesize $\bar{E}_2$};
\node at (25.5,-1.85) {\footnotesize $\bar{E}_3'$};
\node at (28,-1.85) {\footnotesize $\bar{E}_{m+4}'$};

\node at (30.5,3) {$\overset{g}{\longrightarrow}$};

\draw (33,0) -- (41,0);
\draw (33,2.5) -- (41,2.5);
\draw (33,5) -- (41,5);
\draw (34.5,-1) -- (34.5,6);
\node at (38.5,5.85) {\footnotesize $g_{\ast}(\bar{\Gamma})$};
\node at (38.5,3.35) {\footnotesize $g_{\ast}(\bar{L}_1)$};
\node at (38.5,.85) {\footnotesize $g_{\ast}(\bar{Q})$};
\node at (34.5,6.35) {\footnotesize $g_{\ast}(\bar{L}_2)$};
\end{tikzpicture}
\caption{Birational map $g \circ f^{-1}:\tilde{S} \dashrightarrow \F_1$}\label{fig(3)}
\end{center}
\end{figure}
We consider a $(-1)$-curve $\tilde{C}$ on $\tilde{S}$ such that: 
\begin{align*}
\tilde{C} \sim \tilde{Q} + (m+1)\tilde{F} - \sum_{i=2}^{m+4}\tilde{E}_i. 
\end{align*}
Note that such a $(-1)$-curve $\tilde{C}$ exists by Lemma \ref{lem: (-1)-curve (2)} (2). 
Put $C := \pi_{\ast}(\tilde{C})$. 
\begin{lem}\label{lem(3)}
With the same notations as above, we obtain: 
\begin{align*}
S \setminus \Supp \left( \Gamma_{\tilde{\sf q}} + C + E_1 + \sum_{i=2}^{m+4}E_i' \right) \cong \A^1 \times \A^1_{\ast}. 
\end{align*}
\end{lem}
\begin{proof}
Since $(\tilde{C} \cdot \tilde{E}_1) = (\tilde{C} \cdot \tilde{E}_2') = \dots = (\tilde{C} \cdot \tilde{E}_{m+4}') = 0$, 
the union $\tilde{C}+\tilde{E}_1+\sum_{i=2}^{m+4}\tilde{E}_i'$ is disjoint and contractible. 
By contracting this disjoint union, we obtain the birational morphism $f:\tilde{S} \to \bP^2$. 
Then $f_{\ast}(\tilde{Q})$ is an irreducible conic on $\bP^2$ and $f_{\ast}(\tilde{\Gamma}_{\tilde{\sf q}})$ is a tangent line of $f_{\ast}(\tilde{Q})$ at the point $f(\tilde{\sf q})$. 
See also Figure \ref{fig(4)}. 
Hence, we have: 
\begin{align*}
S \setminus \Supp \left( \Gamma_{\tilde{\sf q}} +C + E_1 + \sum_{i=2}^{m+4}E_i' \right) 
&\cong \tilde{S} \setminus \Supp \left( \tilde{Q} + \tilde{\Gamma}_{\tilde{\sf q}} + \tilde{C} + \tilde{E}_1 + \sum_{i=2}^{m+4}\tilde{E}_i' \right) \\
&\cong \bP^2 \setminus \Supp \left(f_{\ast}(\tilde{Q}) + f_{\ast}(\tilde{\Gamma}_{\tilde{\sf q}}) \right) \\
&\cong \A^1 \times \A^1_{\ast}.
\end{align*}
This completes the proof. 
\end{proof}
\begin{figure}[t]
\begin{center}
\begin{tikzpicture}[scale=0.35]
\draw (0,0) -- (13,0);
\node at (.35,6.35) {\footnotesize $\tilde{\Gamma}_{\tilde{\sf q}}$};
\node at (.35,.85) {\footnotesize $\tilde{Q}$};
\draw (0,5.5) 
.. controls (1,5.5) and (1,.1) .. (2.5,0)
.. controls (4,.1) and (4,5.5) .. (5,5.5)
.. controls (5,5.5) and (12,5.5) .. (13,5.5);
\draw (6.5,6) -- (6,3.5);
\draw (9,-1) -- (8.5,1.5);
\node at (10.25,1) {\footnotesize $\cdots$};
\draw (12.5,-1) -- (12,1.5);
\draw (6,-1) 
.. controls (7,3.5) and (9,4.4) .. (9.5,4.5)
.. controls (10,4.5) and (12,4.5) .. (13,4.5);
\node at (11.5,3.6) {\footnotesize $\tilde{C}$};
\node at (6.5,6.35) {\footnotesize $\tilde{E}_1$};
\node at (9,-1.85) {\footnotesize $\tilde{E}_2'$};
\node at (12.5,-1.85) {\footnotesize $\tilde{E}_{m+4}'$};

\node at (14.5,3) {$\overset{f}{\longrightarrow}$};

\draw (16,5) -- (24,5);
\draw (16,.5) 
.. controls (17,.5) and (17,4.9) .. (19,5)
.. controls (21,4.9) and (21,.5) .. (22,.5);
.. controls (22,.5) and (23,.5) .. (24,.5);

\node at (23.5,.75) {\footnotesize $f_{\ast}(\tilde{Q})$};
\node at (22,5.85) {\footnotesize $f_{\ast}(\tilde{\Gamma})$};
\end{tikzpicture}
\caption{Birational morphism $f:\tilde{S} \to \bP^2$}\label{fig(4)}
\end{center}
\end{figure}
\section{Proof of Theorem \ref{thm: main}}\label{3}
In this section, we prove Theorem \ref{thm: main}. 
Let $S := S_m^{m+4}$ be a log del Pezzo surface obtained by the blow-up of the weighted projective plane $\bP(1,1,m)$ at $m+4$ points in general position, where $m \ge 2$. 
Note that $S$ has a unique singular point ${\sf p}$ of type $\frac{1}{m}(1,1)$. 
Let $\pi:\tilde{S} \to S$ be the minimal resolution at ${\sf p}$, and let $\tilde{Q}$ be the reduced exceptional curve of $\pi$. 
Note that $\tilde{Q}$ is a single $(-m)$-curve on $\tilde{S}$. 

Let $H$ be an ample $\Q$-divisor on $S$ and let $\mu_H$ be the Fujita invariant of $H$. 
We may assume that $\mu_H=1$. 
Indeed, put $H' := \frac{1}{\mu_H}H$. Since $\mu_{H'} = \frac{1}{\mu_H} \cdot \mu_H = 1$ and every $H'$-polar cylinder is an $H$-polar cylinder, it is enough to replace $H$ with $H'$. 
Recall that $H$ is of type either ${\rm B}(r_H)$ or ${\rm C}(\ell_H)$. 
\subsection{Type ${\rm B}(r_H)$}\label{3-1}
In this subsection, we shall prove Theorem \ref{thm: main} (1). 
We keep the notation from the beginning of Section \ref{3} and assume further that $H$ is of type ${\rm B}(r_H)$ and $r_H^{sm}>0$. 
Then we can write: 
\begin{align*}
K_S + \mu_HH \sim_{\Q} \sum_{i=1}^{r_H}a_iL_i,
\end{align*}
where every $a_i$ is a positive number and the union $\sum_{i=1}^{r_H}L_i$ is contractible. 
For simplicity, we set $r := r_H$ and $s := r_H^{sm}$. 
We may assume that $(L_i)^2 = -1$ for $i=1,\dots,s$. 
\smallskip

By the assumption, the contraction $\phi_H:S \to S'$ of the Fujita face $\Delta_H$ is a birational morphism. 
By Proposition \ref{prop}, $S'$ is isomorphic to either $S_{m'}^{n'}$ or $\overline{S}_{m'}^{m'+1}$, where $m' := m-(r-s)$ and $n' := m+4-r$.  
We consider the following two cases separately. 
\medskip

\noindent
{\bf Case 1:} ($S' \cong S_{m'}^{n'}$). 
In this case, $S'$ is obtained by blowing-up $\bP(1,1,m')$ at $n'$ points in general position. 
Hence, there exists a birational morphism $\sigma: S' \to \bP(1,1,m')$. 
The morphism $\sigma \circ \phi_H$ induces the birational morphism $\psi': \tilde{S} \to \F_{m'}$. 
Then the composition $\tilde{S} \overset{\psi'}{\to} \F_{m'} \to \bP^1$ of $\psi'$ and the natural $\bP^1$-bundle $\F_{m'} \to \bP^1$ is a $\bP^1$-fibration with exactly $m+4$ singular fibers $\tilde{F}_1,\dots,\tilde{F}_{m+4}$, which are degenerate conics. 
Hence, for $i=1,\dots,m+4$, there exist two $(-1)$-curves $\tilde{E}_i$ and $\tilde{E}_i'$ on $\tilde{S}$ such that $\tilde{F}_i = \tilde{E}_i + \tilde{E}_i'$ and $(\tilde{E}_i \cdot \tilde{Q}) = 0$. 
Letting $\psi:\tilde{S} \to \F_m$ be the contraction of $\tilde{E}_1,\dots,\tilde{E}_{m+4}$, the composition $\varphi: \tilde{S} \overset{\psi}{\to} \F_m \to \bP^1$ of $\psi$ and the natural $\bP^1$-bundle $\F_m \to \bP^1$ is a $\bP^1$-fibration with a section $\tilde{Q}$ and exactly $m+4$ singular fibers $\tilde{F}_1,\dots,\tilde{F}_{m+4}$. 
See also the following commutative diagram: 
\begin{align*}
\xymatrix@C=36pt@R=18pt{
S \ar[d]_{\phi_H} & \tilde{S} \ar[l]_{\pi} \ar[dd]_{\psi'} \ar@{=}[r] & \tilde{S} \ar[dd]_{\psi} \ar@/^24pt/[ddd]^{\varphi} \\
S_{m'}^{n'} \ar[d]_{\sigma} && \\
\bP(1,1,m') & \F_{m'} \ar[l] \ar[d] & \F_m \ar[d] \\
& \bP^1 \ar@{=}[r] & \bP^1
}
\end{align*}
Note that $\pi^{-1}_{\ast}(L_1),\dots,\pi^{-1}_{\ast}(L_r)$ are fiber components of $\varphi$. 
Hence, we may assume that: 
\begin{align*}
\pi^{-1}_{\ast}(L_i) = \left\{\begin{array}{cc}\tilde{E}_i & i \le s \\ \tilde{E}_i' & i > s\end{array}\right..
\end{align*}
\smallskip

Let $\tilde{F}$ be a general fiber of $\varphi$, and let $\tilde{\sf q}$ and $\tilde{\Gamma}_{\tilde{\sf q}}$ be the same as in Lemma \ref{curve}. 
Recall that:
\begin{align*}
\tilde{\Gamma}_{\tilde{\sf q}} \sim \tilde{Q} + (m+2)\tilde{F} - \tilde{E}_1-\dots-\tilde{E}_{m+4}
\end{align*}
and $I_{\tilde{\sf q}}(\tilde{\Gamma}_{\tilde{\sf q}},\tilde{Q}) = 2$. 
Let $\tilde{F}_{\tilde{\sf q}}$ be the fiber of $\varphi$ passing through $\tilde{\sf q}$. 
Note that the configuration of the $\bP^1$-fibration $\varphi$ looks like that in Figure \ref{fig(1)}. 
\smallskip

Without loss of generality, we may assume that $a_1 \ge a_2 \ge \dots \ge a_s$. 
Moreover, we write $a_i := 0$ for every $i=r+1,\dots,m+4$ when $r < m+4$. 
Put $\Gamma_{\tilde{\sf q}} := \pi_{\ast}(\tilde{\Gamma}_{\tilde{\sf q}})$, $F_{\tilde{\sf q}} := \pi_{\ast}(\tilde{F}_{\tilde{\sf q}})$, $E_i := \pi_{\ast}(\tilde{E}_i)$ and $E_i' := \pi_{\ast}(\tilde{E}_i')$ for $i=1,\dots ,m+4$. 
Note that: 
\begin{align*}
H \sim_{\Q} (m+2)F_{\tilde{\sf q}} - \sum_{i=1}^s(1-a_i)E_i - \sum_{j=s+1}^{m+4} (E_j - a_jE_j'), 
\end{align*}
where $\sum_{j=s+1}^{m+4} (E_j - a_jE_j') = 0$ provided $s=m+4$. 
We consider the following four subcases separately. 
\smallskip

\noindent
{\bf Subcase 1-1:} ($s=m+4$). 
In this subcase, let $\varepsilon$ be a positive number satisfying $\varepsilon < a_s = a_{m+4}$, and let $D$ be the effective $\Q$-divisor on $S$ defined by: 
\begin{align*}
D := (1-a_{m+4}+\varepsilon) \Gamma_{\tilde{\sf q}} + (m+2)(a_{m+4}-\varepsilon)F_{\tilde{\sf q}} +\sum_{i=1}^{m+4}(a_i-a_{m+4}+\varepsilon)E_i. 
\end{align*}
By Lemma \ref{lem(1)}, we know $S \setminus \Supp(D) \cong \A^1 \times \A^1_{\ast}$. 
Moreover, since $\Gamma_{\tilde{\sf q}} \sim_{\Q} (m+2)F_{\tilde{\sf q}} - \sum_{i=1}^{m+4}E_i$, we have: 
\begin{align*}
D &\sim_{\Q} (1-a_{m+4}+\varepsilon) \left\{ (m+2)F_{\tilde{\sf q}} -\sum_{i=1}^{m+4}E_i\right\} + (m+2)(a_{m+4}-\varepsilon)F_{\tilde{\sf q}} + \sum_{i=1}^{m+4}(a_i-a_{m+4}+\varepsilon)E_i \\
&\sim_{\Q} (m+2)F_{\tilde{\sf q}} - \sum _{i=1}^{m+4}(1-a_i)E_i \\
&\sim_{\Q} H. 
\end{align*}
Hence, $S \setminus \Supp(D)$ is an $H$-polar cylinder. 
\smallskip

\noindent
{\bf Subcase 1-2:} ($2 < s < m+4$). 
In this subcase, let $\varepsilon$ be a positive number satisfying $\varepsilon < a_s$ and let $D$ be the effective $\Q$-divisor on $S$ defined by: 
\begin{align*}
D := (1-a_s+\varepsilon) \Gamma_{\tilde{\sf q}} + (s-2)(a_s-\varepsilon)F_{\tilde{\sf q}} +\sum_{i=1}^s(a_i-a_s+\varepsilon)E_i + \sum _{j=s+1}^{m+4}(a_s+a_j-\varepsilon)E_j'. 
\end{align*}
By Lemma \ref{lem(1)}, we know $S \setminus \Supp(D) \cong \A^1 \times \A^1_{\ast}$. 
Moreover, since $\Gamma_{\tilde{\sf q}} \sim_{\Q} (m+2)F_{\tilde{\sf q}} - \sum_{i=1}^{m+4}E_i$ and $E_j' \sim_{\Q} F_{\tilde{\sf q}} - E_j$ for $j=s+1,\dots ,m+4$, we have: 
\begin{align*}
D &\sim_{\Q} (1-a_s+\varepsilon) \left\{ (m+2)F_{\tilde{\sf q}} -\sum_{i=1}^sE_i -\sum_{j=s+1}^{m+4}E_j \right\} \\
&\qquad+ (s-2)(a_s-\varepsilon)F_{\tilde{\sf q}} + \sum_{i=1}^s(a_i-a_s+\varepsilon)E_i + \sum _{j=s+1}^{m+4}\left\{ (a_s-\varepsilon)(F_{\tilde{\sf q}}-E_j) + a_j E_j' \right\} \\
&\sim_{\Q} (m+2)F_{\tilde{\sf q}} - \sum_{i=1}^s(1-a_i)E_i - \sum_{j=s+1}^{m+4} (E_j - a_jE_j') \\
&\sim_{\Q} H. 
\end{align*}
Hence, $S \setminus \Supp(D)$ is an $H$-polar cylinder. 
\smallskip

\noindent
{\bf Subcase 1-3:} ($s=2$). 
In this subcase, we take a $0$-curve $\tilde{C}_{\tilde{\sf q}}$ on $\tilde{S}$ passing through $\tilde{\sf q}$ such that
\begin{align*}
\tilde{C}_{\tilde{\sf q}} \sim \tilde{Q} + m\tilde{F}_{\tilde{\sf q}} - \sum_{i=3}^{m+4}\tilde{E}_i. 
\end{align*}
Note that such a $0$-curve $\tilde{C}_{\tilde{\sf q}}$ exists by Lemma \ref{curve(2)}. 
Put $C_{\tilde{\sf q}} := \pi_{\ast}(\tilde{C}_{\tilde{\sf q}})$. 
Let $\varepsilon$ be a positive rational number satisfying $2\varepsilon < \min\{ 1,a_2\}$, and let $D$ be the effective $\Q$-divisor on $S$ defined by: 
\begin{align*}
D := (1-2\varepsilon) \Gamma_{\tilde{\sf q}} + 2\varepsilon F_{\tilde{\sf q}} + \varepsilon C_{\tilde{\sf q}} + (a_1-2\varepsilon)E_1  + (a_2-2\varepsilon)E_2 + \sum_{i=3}^{m+4}(a_i+\varepsilon)E_i'. 
\end{align*}
By Lemma \ref{lem(2)},  we know $S \setminus \Supp(D) \cong \A^1 \times \A^1_{\ast \ast}$. 
Moreover, since $\Gamma_{\tilde{\sf q}} \sim_{\Q} (m+2)F_{\tilde{\sf q}} - \sum_{i=1}^{m+4}E_i$, $C_{\tilde{\sf q}} \sim_{\Q} mF_{\tilde{\sf q}} - \sum_{i=3}^{m+4}E_i'$ and $E_i' \sim_{\Q} F_{\tilde{\sf q}} - E_i$ for $i=3,\dots ,m+4$, we have: 
\begin{align*}
D &\sim_{\Q} (1-2\varepsilon)\left\{ (m+2)F_{\tilde{\sf q}} - E_1 - E_2 - \sum_{i=3}^{m+4}E_i \right\} + 2\varepsilon F_{\tilde{\sf q}} \\
&\qquad+ \varepsilon \left( m F_{\tilde{\sf q}} - \sum_{i=3}^{m+4}E_i \right) + (a_1-\varepsilon)E_1 + (a_2-\varepsilon)E_2 + \sum_{i=3}^{m+4}\left\{ \varepsilon(F_{\tilde{\sf q}}-E_i) + a_iE_i'\right\} \\
&\sim_{\Q} (m+2) F_{\tilde{\sf q}} - (1-a_1)E_1 - (1-a_2)E_2 - \sum _{i=3}^{m+4}(E_i-a_iE_i') \\
&\sim_{\Q} H. 
\end{align*}
Hence, $S \setminus \Supp(D)$ is an $H$-polar cylinder. 
\smallskip

\noindent
{\bf Subcase 1-4:} ($s=1$). 
In this case, we take a $(-1)$-curve $\tilde{C}$ on $\tilde{S}$ such that:
\begin{align*}
\tilde{C} \sim \tilde{Q} + (m+1)\tilde{F} - \sum_{i=2}^{m+4}\tilde{E}_i. 
\end{align*}
Note that such a $(-1)$-curve $\tilde{C}$ exists by Lemma \ref{lem: (-1)-curve (2)} (2). 
Put $C := \pi_{\ast}(\tilde{C})$. 
Let $\varepsilon$ be a positive rational number satisfying $2\varepsilon < \min\{1,a_1\}$, and let $D$ be the effective $\Q$-divisor on $S$ defined by: 
\begin{align*}
D := (1-2\varepsilon) \Gamma_{\tilde{\sf q}} + \varepsilon C + (a_1-2\varepsilon)E_1 + \sum_{i=2}^{m+4}(a_i+\varepsilon)E_i'. 
\end{align*}
By Lemma \ref{lem(3)}, we know $S \setminus \Supp(D) \cong \A^1 \times \A^1_{\ast}$. 
Moreover, since $\Gamma_{\tilde{\sf q}} \sim_{\Q} (m+2)F_{\tilde{\sf q}} - E_1 - \sum_{i=2}^{m+4}E_i$, $C \sim_{\Q} (m+1)F_{\tilde{\sf q}} - \sum_{i=2}^{m+4}E_i'$ and $E_i' \sim_{\Q} F_{\tilde{\sf q}} - E_i$ for $i=2,\dots ,m+4$, we have: 
\begin{align*}
D &\sim_{\Q} (1-2\varepsilon)\left\{ (m+2)F_{\tilde{\sf q}} - E_1- \sum_{i=2}^{m+4}E_i \right\} + \varepsilon \left( (m+1) F_{\tilde{\sf q}} - \sum_{i=2}^{m+4}E_i \right) \\
&\qquad+ (a_1-2\varepsilon)E_1 + \sum_{i=2}^{m+4}\left\{\varepsilon(F_{\tilde{\sf q}}-E_i) +a_iE_i'\right\} \\
&\sim_{\Q} (m+2)F_{\tilde{\sf q}} - (1-a_1)E_1 - \sum _{i=2}^{m+4}(E_i - a_iE_i') \\
&\sim_{\Q} H. 
\end{align*}
Hence, $S \setminus \Supp(D)$ is an $H$-polar cylinder. 
\medskip

\noindent
{\bf Case 2:} ($S' \cong \overline{S}_{m'}^{m'+1}$). 
In this case, we note that $\phi_H$ can be factored as $\alpha'': S \to S''$ and $\alpha': S'' \to S'$ such that $S'' \cong S_{m'}^{m'+1}$. 
Hence, we know $s=4$. 
Without loss of generality, we may assume that $a_1 \ge a_2 \ge a_3 \ge a_4 (>0)$. 
Let $\beta'':S \to \overline{S}''$ be the contraction of $L_1,\dots,L_r$ other than $L_4$, and let $\beta':\overline{S}'' \to \overline{S}'$ be the contraction of $L_4'' := \beta''_{\ast}(L_4)$. Then $\overline{S}' = S'$ because $\alpha' \circ \alpha'' = \phi_H = \beta' \circ \beta''$. 
In particular, $\overline{S}'' \cong S_{m'}^{m'+1}$. 
Since $\overline{S}''$ is obtained by blowing-up $\bP(1,1,m')$ at $m'+1$ points in general position, there exists a birational morphism $\sigma: \overline{S}'' \to \bP(1,1,m')$. 
The morphism $\sigma \circ \beta''$ induces the birational morphism $\psi': \tilde{S} \to \F_{m'}$. 
Then the composition $\tilde{S} \overset{\psi'}{\to} \F_{m'} \to \bP^1$ of $\psi'$ and the natural $\bP^1$-bundle $\F_{m'} \to \bP^1$ is a $\bP^1$-fibration with exactly $m+4$ singular fibers $\tilde{F}_1,\dots,\tilde{F}_{m+4}$, which are degenerate conics. 
Hence, for $i=1,\dots,m+4$, there exist two $(-1)$-curves $\tilde{E}_i$ and $\tilde{E}_i'$ on $\tilde{S}$ such that $\tilde{F}_i = \tilde{E}_i + \tilde{E}_i'$ and $(\tilde{E}_i \cdot \tilde{Q}) = 0$. 
Letting $\psi:\tilde{S} \to \F_m$ be the contraction of $\tilde{E}_1,\dots,\tilde{E}_{m+4}$, the composition $\varphi: \tilde{S} \overset{\psi}{\to} \F_m \to \bP^1$ of $\psi$ and the natural $\bP^1$-bundle $\F_m \to \bP^1$ is a $\bP^1$-fibration with a section $\tilde{Q}$ and exactly $m+4$ singular fibers $\tilde{F}_1,\dots,\tilde{F}_{m+4}$. 
See also the following commutative diagram: 
\begin{align*}
\xymatrix@C=36pt@R=18pt{
&S \ar[ld]_{\phi_H} \ar[d]_{\beta''} & \tilde{S} \ar[l]_{\pi} \ar[dd]_{\psi'} \ar@{=}[r] & \tilde{S} \ar[dd]_{\psi} \ar@/^24pt/[ddd]^{\varphi} \\
\overline{S}_{m'}^{m'+1} & S_{m'}^{m'+1} \ar[l]^{\beta'} \ar[d]_{\sigma} && \\
&\bP(1,1,m') & \F_{m'} \ar[l] \ar[d] & \F_m \ar[d] \\
&& \bP^1 \ar@{=}[r] & \bP^1
}
\end{align*}
Note that $\pi^{-1}_{\ast}(L_1),\dots,\pi^{-1}_{\ast}(L_r)$ except for $\pi^{-1}_{\ast}(L_4)$ are fiber components of $\varphi$. 
Hence, we may assume that: 
\begin{align*}
\pi^{-1}_{\ast}(L_i) = \left\{\begin{array}{cc}\tilde{E}_i & i \le 3 \\ \tilde{E}_i' & i > 4\end{array}\right..
\end{align*}
Then we notice:
\begin{align*}
\tilde{L}_4 := \pi^{-1}_{\ast}(L_4) \sim \tilde{Q} + m \tilde{F} - \tilde{E}_4 - \sum_{j=5}^{m+4}\tilde{E}_j.
\end{align*} 
\smallskip

Let $\tilde{F}$ be a general fiber of $\varphi$, and let $\tilde{\sf q}$ and $\tilde{\Gamma}_{\tilde{\sf q}}$ be the same as in Lemma \ref{curve}. 
Recall that:
\begin{align*}
\tilde{\Gamma}_{\tilde{\sf q}} \sim \tilde{Q} + (m+2)\tilde{F} - (\tilde{E}_1+\dots+\tilde{E}_{m+4})
\end{align*}
and $I_{\tilde{\sf q}}(\tilde{\Gamma}_{\tilde{\sf q}},\tilde{Q}) = 2$. 
Let $\tilde{F}_{\tilde{\sf q}}$ be the fiber of $\varphi$ passing through $\tilde{\sf q}$. 
Note that the configuration of the $\bP^1$-fibration $\varphi$ looks like that in Figure \ref{fig(6)}. 
\medskip

We note $1-a_4>0$ because $s=4$. 
Moreover, we write $a_i := 0$ for every $i=r+1,\dots,m+4$ when $r < m+4$. 
Set $\Gamma_{\tilde{\sf q}} := \pi_{\ast}(\tilde{\Gamma}_{\tilde{\sf q}})$, $F_{\tilde{\sf q}} := \pi_{\ast}(\tilde{F}_{\tilde{\sf q}})$, $E_i := \pi_{\ast}(\tilde{E}_i)$ and $E_i' := \pi_{\ast}(\tilde{E}_i')$ for $i=1,\dots,m+4$. 

Let $\varepsilon$ be a positive number satisfying $\varepsilon < \min\{a_4,1-a_4\}$, and let $D$ be the effective $\Q$-divisor on $S$ defined by: 
\begin{align*}
D := (1-a_4+\varepsilon)\Gamma_{\tilde{\sf q}} + \sum_{i=1}^3(a_i-a_4+\varepsilon)E_i + \varepsilon L_4 + 2(a_4-\varepsilon)E_4' + \sum_{j=5}^{m+4}\left\{a_j+2(a_4-\varepsilon)\right\}E_j'. 
\end{align*}
Since the configuration of the $\bP^1$-fibration $\varphi: \tilde{S} \to \bP^1$ looks like that in Figure \ref{fig(6)}, we know that $S \setminus \Supp(D) \cong \A^1 \times \A^1_{\ast}$ (cf.\ Lemma \ref{lem(3)}). 
Moreover, since $\Gamma_{\tilde{\sf q}} \sim_{\Q} (m+2)F_{\tilde{\sf q}} - \sum_{i=1}^{m+4}E_i$, $L_4 \sim_{\Q} mF_{\tilde{\sf q}} - \sum_{i=4}^{m+4}E_i$, and $E_i' \sim_{\Q} F_{\tilde{\sf q}} - E_i$ for $i=4,\dots,m+4$, we have: 
\begin{align*}
D &\sim_{\Q} (1-a_4+\varepsilon) \left\{ (m+2)F_{\tilde{\sf q}} - \sum_{i=1}^3E_i - E_4 - \sum_{j=5}^{m+4}E_j \right\} \\
&\qquad + \sum_{i=1}^3(a_i-a_4+\varepsilon)E_i 
+\varepsilon \left( mF_{\tilde{\sf q}} - E_4 - \sum_{j=5}^{m+4}E_j\right) \\
&\qquad\qquad+ 2(a_4-\varepsilon)(F_{\tilde{\sf q}} -E_4) + \sum_{j=5}^{m+4}2(a_4-\varepsilon) (F_{\tilde{\sf q}} - E_j) + \sum_{j=5}^ra_jE_j' \\
&\sim_{\Q} \left\{ (m+2)F_{\tilde{\sf q}} - \sum_{i=1}^3E_i - E_4 - \sum_{j=5}^{m+4}E_j \right\} + \sum_{i=1}^3a_iE_i + a_4\left( mF_{\tilde{\sf q}} - E_4 - \sum_{j=5}^rE_j \right) + \sum_{j=5}^ra_jE_j' \\
&\sim_{\Q} -K_S + \sum_{i=1}^3a_iL_i + a_4L_4 + \sum_{j=5}^ra_jL_j \\
&\sim_{\Q} H. 
\end{align*}
Hence, $S \setminus \Supp(D)$ is an $H$-polar cylinder. 
\begin{figure}[t]
\begin{center}
\begin{tikzpicture}[scale=0.7]
\draw (0,0) -- (16,0);
\node at (.35,6) {\footnotesize $\tilde{\Gamma}_{\tilde{\sf q}}$};
\node at (-.65,0) {\footnotesize $\tilde{Q}$};
\draw (0,5.5) 
.. controls (1,5.5) and (1,.1) .. (2.5,0)
.. controls (4,.1) and (4,5.5) .. (5,5.5)
.. controls (5,5.5) and (12,5.5) .. (16,5.5);
\draw (6,6) -- (5.5,3.5);
\draw (7.25,6) -- (6.75,3.5);
\draw (8.5,6) -- (8,3.5);
\draw (10,6) -- (9.5,3.5);
\draw (11.5,-1) -- (11,1.5);
\draw (12.75,-1) -- (12.25,1.5);
\node at (13.25,1) {\footnotesize $\cdots$};
\draw (14.75,-1) -- (14.25,1.5);
\node at (6,6.35) {\footnotesize $\tilde{E}_1$};
\node at (7.25,6.35) {\footnotesize $\tilde{E}_2$};
\node at (8.5,6.35) {\footnotesize $\tilde{E}_3$};
\node at (10,6.35) {\footnotesize $\tilde{L}_4$};
\node at (11.5,-1.85) {\footnotesize $\tilde{E}_4'$};
\node at (12.75,-1.85) {\footnotesize $\tilde{E}_5'$};
\node at (14.75,-1.85) {\footnotesize $\tilde{E}_{m+4}'$};

\end{tikzpicture}
\caption{Configuration of the $\bP^1$-fibration $\varphi: \tilde{S} \to \bP^1$}\label{fig(6)}
\end{center}
\end{figure}
\medskip

The proof of Theorem \ref{thm: main} (1) is thus completed. 
\subsection{Type ${\rm C}(\ell_H)$}\label{3-2}
In this subsection, we shall prove Theorem \ref{thm: main} (2). 
We keep the notation from the beginning of Section \ref{3} and assume further that $H$ is of type ${\rm C}(\ell_H)$. 
Then we can write: 
\begin{align*}
K_S + \mu_HH \sim_{\Q} aB + \sum_{i=1}^{m+3}a_iL_i,
\end{align*}
where $a$ is a positive number, every $a_i$ is a non-negative number, $B$ is a $0$-curve and the union $\sum_{i=1}^{m+3}L_i$ is contractible. 
For simplicity, we set $r := \ell_H$ and $s := \ell_H^{sm}$. 
Then the contraction $\phi_H:S \to \bP^1$ associated to the Fujita face $\Delta_H$ is a $\bP^1$-fibration. 
Since a general fiber of $\phi_H$ does not pass through ${\sf p}$, $\phi_H \circ \pi:\tilde{S} \to \bP^1$ is also a $\bP^1$-fibration. 
Then the $\bP^1$-fibration $\phi_H \circ \pi:\tilde{S} \to \bP^1$ admits exactly five singular fibers $\tilde{B}_1$, $\tilde{B}_2$, $\tilde{B}_3$, $\tilde{B}_4$ and $\tilde{B}_5$ (see \S \ref{2-1}). 
Set $\tilde{L}_i := \pi^{-1}_{\ast}(L_i)$ for $i=1,\dots,m+3$. Note that they are $(-1)$-curves on $\tilde{S}$. 
Since $\tilde{L}_1,\dots,\tilde{L}_{m+3}$ and $\tilde{Q}$ are fiber components of $\phi_H \circ \pi$, there exist $(-1)$-curves $\tilde{L}_1'$, $\tilde{L}_2'$, $\tilde{L}_3'$, $\tilde{L}_4'$ and $\tilde{L}_{m+4}$ on $\tilde{S}$ such that we may assume that all singular fibers of $\phi_H \circ \pi$ are given by $\tilde{B}_i = \tilde{L}_i + \tilde{L}_i'$ for $i=1,2,3,4$ and $\tilde{B}_5 = \tilde{Q} + \tilde{L}_5 + \dots + \tilde{L}_{m+3} + \tilde{L}_{m+4}$. 
Let $\tau_1:\tilde{S} \to \hat{S}$ be the contraction of $\tilde{L}_1,\dots,\tilde{L}_{m+3}$. 
Since $\hat{Q} := \tau_{1,\ast}(\tilde{Q})$ is a $(-1)$-curve on $\hat{S}$ by virtue of $(-K_{\hat{S}})^2 = (-K_{\tilde{S}})^2 + (m+3) = 7$ and $(\hat{Q})^2 = (\tilde{Q})^2 + (m+3-4) = -1$, $\hat{S}$ is a smooth del Pezzo surface of degree $7$. 
Let $\pi_1: \hat{S} \to \check{S}$ be the contraction of $\hat{Q}$. 
Then either $\check{S} \cong \F_1$ or $\check{S} \cong \bP^1 \times \bP^1$. 
We consider the following two cases separately. 
\smallskip

\noindent
{\bf Case 1:} ($\check{S} \cong \F_1$). 
In this case, there exists a unique $(-1)$-curve $\hat{L}$ on $\hat{S}$ such that $(\hat{Q} \cdot \hat{L}) = 0$. 
Hence, we obtain the contraction $\tau_2:\hat{S} \to \F_1$ of $\hat{L}$. 
We know that there exist exactly $m+4$ singular fibers $\tilde{F}_1,\dots,\tilde{F}_{m+4}$ of the $\bP^1$-fibration $\tilde{S} \overset{\tau_2 \circ \tau_1}{\longrightarrow} \F_1 \to \bP^1$ and they are degenerate conics, where $\F_1 \to \bP^1$ is the natural $\bP^1$-bundle. 
Hence, for $i=1,\dots,m+4$, there exist two $(-1)$-curves $\tilde{E}_i$ and $\tilde{E}_i'$ on $\tilde{S}$ such that $\tilde{F}_i = \tilde{E}_i + \tilde{E}_i'$ and $(\tilde{E}_i \cdot \tilde{Q}) = 0$. 
Considering the contraction $\psi:\tilde{S} \to \F_m$ of $\tilde{E}_1,\dots,\tilde{E}_{m+4}$, we have a $\bP^1$-fibration $\varphi:\tilde{S} \overset{\psi}{\to} \F_m \to \bP^1$ such that $\tilde{Q}$ is a section and every $\tilde{L}_i$ is a fiber component of $\varphi$, where $\F_m \to \bP^1$ is the natural $\bP^1$-bundle. 
See also the following commutative diagram: 
\begin{align*}
\xymatrix@C=24pt@R=24pt{
\bP^1&S \ar[l]_{\phi_H} & \tilde{S} \ar[l]_{\pi} \ar[d]_{\tau_1} \ar@{=}[r] & \tilde{S} \ar[dd]_{\psi} \ar@/^24pt/[ddd]^{\varphi} \\
&\check{S} & \hat{S} \ar[l]_{\pi_1} \ar[d]_{\tau_2} && \\
&& \F_1 \ar[d] & \F_m \ar[d] \\
&& \bP^1 \ar@{=}[r] & \bP^1
}
\end{align*}
We may assume that: 
\begin{align*}
\tilde{L}_i = \left\{ \begin{array}{cl}\tilde{E}_i & \text{if $i \le 4$} \\ \tilde{E}_i' & \text{if $i \ge 5$} \end{array}\right. 
\end{align*}
for $i=1,\dots ,m+4$. 
Let $\tilde{F}$ be a general fiber of $\varphi$. 
We note that: 
\begin{align*}
\pi^{-1}_{\ast}(B) &\sim \tilde{Q} + \tilde{E}_5' + \dots +\tilde{E}_{m+4}' \\
&\sim \tilde{Q} + m\tilde{F} - \tilde{E}_5 - \dots - \tilde{E}_{m+4}.
\end{align*}
%
Put $F := \pi_{\ast}(\tilde{F})$, $E_i := \pi_{\ast}(\tilde{E}_i)$ and $E_i' := \pi_{\ast}(\tilde{E}_i')$ for $i=1,\dots ,m+4$. 
\smallskip

Suppose that $s>0$. 
By similar arguments on Subcases 1-2, 1-3 and 1-4 according to $s$, there exists an effective $\Q$-divisor $D'$ on $S$ such that $D' \sim_{\Q} H - aB$, $S \setminus \Supp(D') \cong \A^1 \times \A^1_{\ast}$, and $E_5' \cup \dots \cup E_{m+4}' \subseteq \Supp(D')$. 
By virtue of $B \sim_{\Q} E_5' + \dots + E_{m+4}'$, we know that $D := D' + a\sum_{i=5}^{m+4}E_i'$ is effective $\Q$-divisor satisfying $D \sim_{\Q} H$ and $S \setminus \Supp (D) = S \setminus \Supp(D') \cong \A^1 \times \A^1_{\ast}$. 
Hence, $S \setminus \Supp (D)$ is an $H$-polar cylinder. 
\smallskip

In what follows, we assume that $s=0$ and $a>3$. Put $L_i' := \pi_{\ast}(\tilde{L}_i)$ for $i=1,2,3,4$. Moreover, we write $a_{m+4} := 0$. 
Let $\varepsilon$ be a positive number satisfying $4\varepsilon < a-3$, and let $D$ be the effective $\Q$-divisor on $S$ defined by: 
\begin{align*}
D := 2F + \sum_{i=1}^4 \left\{ \varepsilon E_i + (1+\varepsilon)L_i'\right\} + \sum_{j=5}^{m+4}( a+a_j-3 - 4\varepsilon)E_j'. 
\end{align*}
Since $(\tilde{F} \cdot \pi^{-1}_{\ast}(B)) = 1$, the configuration of the $\bP^1$-fibration $\phi_H \circ \pi: \tilde{S} \to \bP^1$ looks like that in Figure \ref{fig(5)}. 
Thus, we know that $S \setminus \Supp(D) \cong \A^1 \times (\A^1 \setminus \{\text{$4$ points}\})$. 
Moreover, since $L_i' \sim_{\Q} B-E_i \sim_{\Q} mF - (E_i+E_5+\dots+E_{m+4})$ for $i=1,2,3,4$, we have: 
\begin{align*}
D &\sim_{\Q} 2F + \sum_{i=1}^4  \left\{ \varepsilon E_i + (1+\varepsilon)(mF-E_i - E_5-\dots-E_{m+4}) \right\} \\
&\qquad+ (a-3 - 4\varepsilon)(mF - E_5 - \dots - E_{m+4}) + \sum_{j=5}^{m+4}a_jE_j'  \\
&\sim_{\Q} \left\{(m+2)F -(E_1 + \dots + E_{m+4})\right\} + a(mF - E_5-\dots -E_{m+4}) + \sum_{i=1}^40E_i + \sum_{j=5}^{m+4}a_jL_j \\
&\sim_{\Q} -K_S + aB + \sum_{i=1}^{m+3}a_iL_i \\
&\sim_{\Q} H. 
\end{align*}
Hence, $S \setminus \Supp(D)$ is an $H$-polar cylinder. 
\begin{figure}[t]
\begin{center}
\begin{tikzpicture}[scale=0.7]
\draw (2,-1) -- (1,3);
\draw (2,6) -- (1,2);
\draw (3.5,-1) -- (2.5,3);
\draw (3.5,6) -- (2.5,2);
\draw (5,-1) -- (4,3);
\draw (5,6) -- (4,2);
\draw (6.5,-1) -- (5.5,3);
\draw (6.5,6) -- (5.5,2);
\draw (7.5,1.5) -- (9.5,1.5);
\draw (7.5,4.5) -- (9.5,4.5);
\draw[very thick] (8.5,6) -- (8.5,-1);
\node at (8,3.35) {\large $\vdots$};
\node at (8,2.65) {\large $\vdots$};
\draw (0,0) -- (10.5,0);

\node at (-.4,0) {$\tilde{F}$};
\node at (2,-1.45) {$\tilde{L}_1'$};
\node at (3.5,-1.45) {$\tilde{L}_2'$};
\node at (5,-1.45) {$\tilde{L}_3'$};
\node at (6.5,-1.45) {$\tilde{L}_4'$};
\node at (2,6.35) {$\tilde{E}_1$};
\node at (3.5,6.35) {$\tilde{E}_2$};
\node at (5,6.35) {$\tilde{E}_3$};
\node at (6.5,6.35) {$\tilde{E}_4$};
\node at (9.9,1.5) {$\tilde{E}_5'$};
\node at (10.3,4.5) {$\tilde{E}_{m+4}'$};
\end{tikzpicture}
\caption{Configuration of the $\bP^1$-fibration $\phi_H \circ \pi: \tilde{S} \to \bP^1$}\label{fig(5)}
\end{center}
\end{figure}
\medskip

\noindent
{\bf Case 2:} ($\check{S} \cong \bP^1 \times \bP^1$). 
In this case, let $\tau_1':\tilde{S} \to \hat{S}'$ be the contraction of $\tilde{L}_1,\dots,\tilde{L}_3,\tilde{L}_4',\tilde{L}_5,\dots,\tilde{L}_{m+3}$ and $\pi_1':\hat{S}' \to \check{S}'$ be the contraction of $\hat{Q}' := \tau'_{1,\ast}(\tilde{Q})$. 
Then we know that $\check{S}' \cong \F_1$. 
If $s < 4$, then we have: 
\begin{align*}
K_S + \mu_HH \sim_{\Q} aB + a_1L_1 + a_2L_2 + a_3L_3 + a_4\pi_{\ast}(\tilde{L}_4') + \sum_{j=5}^{m+3}a_jL_j, 
\end{align*}
because of $a_4=0$, so that we can reduce the argument to Case 1 by replacing the role of $L_4$ with that of $\pi_{\ast}(\tilde{L}_4')$. 
Thus, we may assume $s=4$ in what follows. 

Since $\hat{S}'$ is a smooth del Pezzo surface of degree $7$, there exists a unique $(-1)$-curve $\hat{L}'$ on $\hat{S}'$ such that $(\hat{Q}' \cdot \hat{L}') = 0$. 
Hence, we obtain the contraction $\tau_2':\hat{S}' \to \F_1$ of $\hat{L}'$. 
We know that there exist exactly $m+4$ singular fibers $\tilde{F}_1,\dots,\tilde{F}_{m+4}$ of the $\bP^1$-fibration $\tilde{S} \overset{\tau_2' \circ \tau_1'}{\longrightarrow} \F_1 \to \bP^1$ and they are degenerate conics, where $\F_1 \to \bP^1$ is the natural $\bP^1$-bundle. 
Hence, for $i=1,\dots,m+4$, there exist two $(-1)$-curves $\tilde{E}_i$ and $\tilde{E}_i'$ on $\tilde{S}$ such that $\tilde{F}_i = \tilde{E}_i + \tilde{E}_i'$ and $(\tilde{E}_i \cdot \tilde{Q}) = 0$. 
Considering the contraction $\psi':\tilde{S} \to \F_m$ of $\tilde{E}_1,\dots,\tilde{E}_{m+4}$, we have a $\bP^1$-fibration $\varphi:\tilde{S} \overset{\psi'}{\to} \F_m \to \bP^1$ such that $\tilde{Q}$ is a section and every $\tilde{L}_i$ with $i \not= 4$, as well as $\tilde{L}_4'$ are fiber components of $\varphi$, where $\F_m \to \bP^1$ is the natural $\bP^1$-bundle. 
See also the following commutative diagram: 
\begin{align*}
\xymatrix@C=24pt@R=24pt{
\bP^1&S \ar[l]_{\phi_H} & \tilde{S} \ar[l]_{\pi} \ar[d]_{\tau_1'} \ar@{=}[r] & \tilde{S} \ar[dd]_{\psi'} \ar@/^24pt/[ddd]^{\varphi} \\
&\check{S}' & \hat{S}' \ar[l]_{\pi_1'} \ar[d]_{\tau_2'} && \\
&& \F_1 \ar[d] & \F_m \ar[d] \\
&& \bP^1 \ar@{=}[r] & \bP^1
}
\end{align*}
We may assume that: 
\begin{align*}
\tilde{L}_i = \left\{ \begin{array}{cl}\tilde{E}_i & \text{if $i \le 3$} \\ \tilde{E}_i' & \text{if $i \ge 5$} \end{array}\right. 
\end{align*}
for $i=1,\dots ,m+4$ and: 
\begin{align*}
\tilde{L}_4' = \tilde{E}_4.
\end{align*}
Let $\tilde{F}$ be a general fiber of $\varphi$ and let $\tilde{\sf q}$ and $\tilde{\Gamma}_{\tilde{\sf q}}$ be the same as in Lemma \ref{curve}. 
Recall that: 
\begin{align*}
\tilde{\Gamma}_{\tilde{\sf q}} \sim \tilde{Q} + (m+2)\tilde{F} - (\tilde{E}_1+\dots+\tilde{E}_{m+4})
\end{align*}
and $I_{\tilde{\sf q}}(\tilde{\Gamma}_{\tilde{\sf q}},\tilde{Q}) = 2$. 
Let $\tilde{F}_{\tilde{\sf q}}$ be the fiber of $\varphi$ passing through $\tilde{\sf q}$. 
We note that: 
\begin{align*}
\pi^{-1}_{\ast}(B) &\sim \tilde{Q} + \tilde{E}_5' + \dots +\tilde{E}_{m+4}' \\
&\sim \tilde{Q} + m\tilde{F}_{\tilde{\sf q}} - \tilde{E}_5 - \dots - \tilde{E}_{m+4}, \\
\tilde{L}_4 &\sim \pi^{-1}_{\ast}(B) - \tilde{E}_4 \\
&\sim \tilde{Q} + m\tilde{F}_{\tilde{\sf q}} - \tilde{E}_4 - \tilde{E}_5 - \dots - \tilde{E}_{m+4}.
\end{align*}
Put $\Gamma_{\tilde{\sf q}} := \pi_{\ast}(\tilde{\Gamma}_{\tilde{\sf q}})$, $F_{\tilde{\sf q}} := \pi_{\ast}(\tilde{F}_{\tilde{\sf q}})$, $E_i := \pi_{\ast}(\tilde{E}_i)$ and $E_i' := \pi_{\ast}(\tilde{E}_i')$ for $i=1,\dots ,m+4$. 

Without loss of generality, we may assume that $a_1 \ge a_2 \ge a_3 \ge a_4 (>0)$. We note $1-a_4>0$ because $s=4$. 
Moreover, we write $a_{m+4} := 0$. 

Let $\varepsilon$ be a positive number satisfying $\varepsilon < \min\{a_4,1-a_4\}$, and let $D$ be the effective $\Q$-divisor on $S$ defined by: 
\begin{align*}
D := (1-a_4+\varepsilon)\Gamma_{\tilde{\sf q}} + \sum_{i=1}^3(a_i-a_4+\varepsilon)E_i + \varepsilon L_4 + 2(a_4-\varepsilon)E_4' + \sum_{j=5}^{m+4}\left\{a+a_j+2(a_4-\varepsilon)\right\}E_j'. 
\end{align*}
Since the configuration of the $\bP^1$-fibration $\varphi: \tilde{S} \to \bP^1$ looks like that in Figure \ref{fig(6)}, we know that $S \setminus \Supp(D) \cong \A^1 \times \A^1_{\ast}$ (cf.\ Lemma \ref{lem(3)}). 
Moreover, since $\Gamma_{\tilde{\sf q}} \sim_{\Q} (m+2)F_{\tilde{\sf q}} - \sum_{i=1}^{m+4}E_i$, $L_4 \sim_{\Q} mF_{\tilde{\sf q}} - \sum_{i=4}^{m+4}E_i$, $B \sim_{\Q} mF_{\tilde{\sf q}} -\sum_{i=5}^{m+4}E_i$ and $E_i' \sim_{\Q} F_{\tilde{\sf q}} - E_i$ for $i=4,\dots,m+4$, we have: 
\begin{align*}
D &\sim_{\Q} (1-a_4+\varepsilon) \left\{ (m+2)F_{\tilde{\sf q}} - \sum_{i=1}^3E_i - E_4 - \sum_{j=5}^{m+4}E_j \right\} \\
&\qquad + \sum_{i=1}^3(a_i-a_4+\varepsilon)E_i 
+\varepsilon \left( mF_{\tilde{\sf q}} - E_4 - \sum_{j=5}^{m+4}E_j\right) \\
&\qquad\qquad+ 2(a_4-\varepsilon)(F_{\tilde{\sf q}} -E_4) + \sum_{j=5}^{m+4}\left\{a+2(a_4-\varepsilon)\right\} (F_{\tilde{\sf q}} - E_j) + \sum_{j=5}^{m+3}a_jE_j' \\
&\sim_{\Q} \left\{ (m+2)F_{\tilde{\sf q}} - \sum_{i=1}^3E_i - E_4 - \sum_{j=5}^{m+4}E_j \right\} + a \left( mF_{\tilde{\sf q}} - \sum_{j=5}^{m+4}E_j\right) \\
&\qquad+ \sum_{i=1}^3a_iE_i + a_4\left( mF_{\tilde{\sf q}} - E_4 - \sum_{j=5}^{m+4}E_j \right) + \sum_{j=5}^{m+3}a_jE_j' \\
&\sim_{\Q} -K_S + aB + \sum_{i=1}^3a_iL_i + a_4L_4 + \sum_{j=5}^{m+3}a_jL_j \\
&\sim_{\Q} H. 
\end{align*}
Hence, $S \setminus \Supp(D)$ is an $H$-polar cylinder. 
\medskip

The proof of Theorem \ref{thm: main} (2) is thus completed. 

\begin{ack}
{\rm The authors would like to thank the anonymous referee for the careful reading of this article and many useful comments. 
The second author was supported by JSPS KAKENHI Grant Numbers JP24K22823 and JP25K17222. The first author was supported by the National Research Foundation of Korea under Grant number RS-2025-00513064, while the third author was supported under Grant numbers RS-2025-00513064 and RS-2026-25506097.}
\end{ack}

\end{document}